\documentclass{amsart}

\usepackage{thesis} 

\begin{document}

\title{\MakeUppercase{\qadjnstitle}}
\date{\today}
\author{\textsc{Aaron Mazel-Gee}}

\begin{abstract}
We prove that various structures on model $\infty$-categories descend to corresponding structures on their localizations: (i) Quillen adjunctions; (ii) two-variable Quillen adjunctions; (iii) monoidal and symmetric monoidal model structures; and (iv) enriched model structures. 
\end{abstract}

\maketitle

\papernum{5}

\setcounter{tocdepth}{1}
\tableofcontents

\setcounter{section}{-1}

\section{Introduction}

\subsection{Presenting structures on localizations of model $\infty$-categories}

\boilerplateintro{\M} In \cite{MIC-sspaces}, we introduced the notion of a \bit{model structure} extending the data of a relative $\infty$-category: just as in Quillen's classical theory of model structures on relative categories, this allows for much more control over manipulations within its localization.\footnote{For the precise definition a model $\infty$-category, we refer the reader to \cite[\sec 1]{MIC-sspaces}.  However, for the present discussion, it suffices to observe that it is simply a direct generalization of the standard definition of a model category.}  For instance, in \cite{MIC-fundthm} we prove that a model structure provides an efficient and computable way of accessing the hom-spaces $\hom_{\loc{\M}{\bW}}(x,y)$.

However, we are not just interested in localizations of relative $\infty$-categories themselves.  For example, adjunctions are an extremely useful structure, and we would therefore like a systematic way of presenting an adjunction on localizations via some structure on overlying relative $\infty$-categories.  The purpose of this paper is to show that model structures on relative $\infty$-categories are not only useful for computations \textit{within} their localizations, but are in fact also useful for presenting structures \textit{on} their localizations.  More precisely, we prove the following sequence of results.\footnote{The precise definitions of Quillen adjunctions and Quillen equivalences are also contained in \cite[\sec 1]{MIC-sspaces}, while the remaining relevant definitions are contained in the body of the present paper.  However, for the present discussion, it likewise suffices to observe that they are all direct generalizations of their classical counterparts.}

\begin{thm*}[\ref{adjunction thm} and \ref{cor quillen equivce}]
A \bit{Quillen adjunction} between model $\infty$-categories induces a canonical adjunction on their localizations.  If this is moreover a \bit{Quillen equivalence}, then the resulting adjunction is an adjoint equivalence.
\end{thm*}

\begin{thm*}[\ref{two-var adjunction thm}]
A \bit{two-variable Quillen adjunction} between model $\infty$-categories induces a canonical two-variable adjunction on their localizations.
\end{thm*}

\begin{thm*}[\ref{fund thm of monoidal model infty-cats} and \ref{fund thm of symm monoidal model infty-cats}]
The localization of a (resp.\! \bit{symmetric}) \bit{monoidal model $\infty$-category} is canonically a closed (resp.\! symmetric) monoidal $\infty$-category.
\end{thm*}

\begin{thm*}[\ref{fund thm of enriched model infty-cats}]
The localization of an \bit{enriched model $\infty$-category} is canonically enriched and bitensored over the localization of the enriching model $\infty$-category.
\end{thm*}
\noindent Along the way, we also develop the foundations of the theory of \bit{homotopy co/limits} in model $\infty$-categories.


\begin{rem}\label{intro rmk subcats of good objs}
Perhaps surprisingly, none of these results depends on the concrete identification of the hom-spaces $\hom_{\loc{\M}{\bW}}(x,y)$ in the localizations of model $\infty$-categories provided in \cite{MIC-fundthm}.  Rather, their proofs all rely on considerations involving subcategories of ``nice'' objects relative to the given structure, for instance the subcategory of cofibrant objects relative to a left Quillen functor.  Such considerations are thus somewhat akin to the theory of ``deformable'' functors of Dwyer--Hirschhorn--Kan--Smith (see \cite{DHKS}, as well as Shulman's excellent synthesis and contextualization \cite{Shulman-hocolims}), but the philosophy can be traced back at least as far as Brown's ``categories of fibrant objects'' (see \cite{BrownKSthesis}).
\end{rem}

\begin{rem}
In the special case of model categories and their \textit{1-categorical} localizations, these results are all quite classical (and fairly easy to prove).\footnote{Given a relative category $(\R,\bW)$, its 1-categorical localization and its $\infty$-categorical localization are closely related: there is a natural functor $\loc{\R}{\bW} \ra \ho(\loc{\R}{\bW}) \simeq \R[\bW^{-1}]$ between them, namely the projection to the homotopy category (see \cite[Remark 1.29]{MIC-rnerves}).  Moreover, all of these structures -- adjunctions, two-variable adjunctions, closed (symmetric) monoidal structures, and enrichments and bitensorings -- descend canonically from $\infty$-categories to their homotopy categories.}  However, the study of \textit{$\infty$-categorical} localizations -- even just of model categories -- is much more subtle, because it requires keeping track of a wealth of coherence data.

The following specializations of our results to model 1-categories (and their $\infty$-categorical localizations) appear in the literature.
\begin{itemize}
\item We proved this special case of the first of the results (regarding Quillen adjunctions) listed above as \cite[\cref{adjns:main theorem}]{adjns}.  (For a detailed history of partial results in this direction, we refer the reader to \cite[\cref{adjns:section history}]{adjns}.)
\item Under a more restrictive definition of a (resp.\! symmetric) monoidal model category in which the unit object is required to be cofibrant (as opposed to unit axiom {\monunitaxiom} of \cref{defn of monoidal model infty-cat}), Lurie proves that its localization admits a canonical (resp.\! symmetric) monoidal structure in \cite[\sec 4.3.1]{LurieHA} (see particularly \cite[Proposition 4.1.3.4]{LurieHA}).  Moreover, under an analogously more restrictive definition of a (resp.\! symmetric) monoidal model $\infty$-category, a canonical (resp.\! symmetric) monoidal structure on its localization likewise follows from this same result.  (See \cref{lurie proves mon or symm mon str for cofibt unit}.)
\end{itemize}
Aside from these, the results of this paper appear to be new, even in the special case of model 1-categories.
\end{rem}

\begin{rem}
Our result \cite[\cref{adjns:main theorem}]{adjns} is founded in point-set considerations, for instance making reference to an explicit ``underlying quasicategory'' functor from relative categories (e.g.\! model categories).  By contrast, the proof of the generalization given here works invariantly, and relies on a crucial result of Gepner--Haugseng--Nikolaus identifying cocartesian fibrations as lax colimits, which appeared almost concurrently to our \cite{adjns}.  (Specifically, the proof of our ``fiberwise localization'' result \cref{fiberwise localization} appeals multiple times to \cite[Theorem 7.4]{GHN}.)  Nevertheless, we hope that our model-specific proof will still carry some value: the techniques used therein seem fairly broadly applicable, and its point-set nature may someday prove useful as well.
\end{rem}

\subsection{Conventions}

\refMIC

\tableofcodenames

\examplecodename

\citelurie \ \luriecodenames

\butinvt \ \seeappendix

\subsection{Outline}

We now provide a more detailed outline of the contents of this paper.

\begin{itemize}

\item In \cref{section qadjns and ho-co-lims and reedy}, we begin by stating our results concerning Quillen adjunctions and Quillen equivalences (\cref{adjunction thm} and \cref{cor quillen equivce}, resp.).  We then develop the rudiments of the theory of homotopy co/limits in model $\infty$-categories, and provide a detailed study of Reedy model structures on functor $\infty$-categories.

\item In \cref{section rel co-cart-fibns and bicart fibns}, we provide some auxiliary material on \textit{relative co/cartesian fibrations} and on \textit{bicartesian fibrations}.  These two enhancements of the theory of co/cartesian fibrations are used in the proofs of the main results of the paper.

\item In \cref{section proofs of qadjns and qeqs}, we prove \cref{adjunction thm} and \cref{cor quillen equivce}.

\item In \cref{section two-var qadjns}, we show that two-variable Quillen adjunctions between model $\infty$-categories present two-variable adjunctions between their localizations.

\item In \cref{section mon and symm mon model infty-cats}, we show that (resp.\! symmetric) monoidal model $\infty$-categories present closed (resp.\! symmetric) monoidal $\infty$-category.

\item In \cref{section enr model infty-cats}, we show that enriched model $\infty$-categories present enriched and bitensored $\infty$-categories.

\end{itemize}

\subsection{Acknowledgments}

We would like to thank Geoffroy Horel and Zhen Lin Low for their helpful comments, as well as the NSF graduate research fellowship program (grant DGE-1106400) for financial support during the time that this work was carried out.

\section{Quillen adjunctions, homotopy co/limits, and Reedy model structures}\label{section qadjns and ho-co-lims and reedy}

Model structures on relative (1- and $\infty$-)categories are extremely useful for making computations \textit{within} their localizations.  However, it can also be quite useful to obtain relationships \textit{between} their localizations.  Perhaps the most important relationship that two $\infty$-categories can share is that of being related by an \textit{adjunction}.  The central result of this section (\cref{adjunction thm}) provides a systematic way of obtaining just such a relationship: a \textit{Quillen adjunction} between model $\infty$-categories induces a canonical \textit{derived adjunction} on their localizations.  As a special case (\cref{cor quillen equivce}), a \textit{Quillen equivalence} induces a \textit{derived equivalence} on localizations.\footnote{Quillen adjunctions and Quillen equivalences are respectively given as Definitions \Cref{sspaces:define quillen adjunction} \and \Cref{sspaces:define quillen equivalence}.  These are completely straightforward generalizations of the model 1-categorical counterparts, and so we do not feel the need to repeat them here.}

This section is organized as follows.  In \cref{subsection qadjns and qeqs}, we state these fundamental theorems regarding Quillen adjunctions and Quillen equivalences.  (However, their proofs will be postponed to \cref{section proofs of qadjns and qeqs}, after we have developed some necessary scaffolding in \cref{section rel co-cart-fibns and bicart fibns}.)  Then, in \cref{subsection ho-co-lims} we study the important special case of \textit{homotopy co/limits}, briefly introducing the \textit{projective} and \textit{injective} model structures.  Finally, in \cref{subsection reedy model str}, we pursue a more in-depth study of the \textit{Reedy} model structure.

\subsection{Quillen adjunctions and Quillen equivalences}\label{subsection qadjns and qeqs}

The classical theory of \textit{derived functors} arose out of a desire to ``correct'' functors between relative categories which do not respect weak equivalences to ones that do.  There, one replaces a given object by a suitable \textit{resolution} -- the nature of which depends both on the context and on the sort of functor which one is attempting to correct -- and then applies the original functor to this resolution, the point being that the functor \textit{does} respect weak equivalences between such ``nice'' objects.

A Quillen adjunction
\[ F : \M \adjarr \N : G \]
between model (1- or $\infty$-)categories is a prototypical and beautifully symmetric example of such a situation.  In general, neither Quillen adjoint will preserve weak equivalences.  However, in this case there are canonical choices for such subcategories of ``nice'' objects: left Quillen functors preserve weak equivalences between cofibrant objects, while right Quillen functors preserve weak equivalences between fibrant objects (see Kenny Brown's lemma (\ref{kenny brown})).  Moreover, the inclusions $(\M^c,\bW_\M^c) \hookra (\M,\bW_\M)$ and $(\N,\bW_\N) \hookla (\N^f,\bW_\N^f)$ induce equivalences
\[ \loc{\M^c}{(\bW_\M^c)} \xra{\sim} \loc{\M}{\bW_\M} \]
and
\[ \loc{\N}{\bW_\N} \xla{\sim} \loc{\N^f}{(\bW_\N^f)} \]
on localizations (see \cref{inclusion of fibts induces equivce on gpd-compln}).  A perfect storm then ensues.

\begin{thm}\label{adjunction thm}
A Quillen adjunction
\[ F : \M \adjarr \N : G \]
of model $\infty$-categories induces a canonical adjunction
\[ \bbL F : \loc{\M}{\bW_\M} \adjarr \loc{\N}{\bW_\N} : \bbR G \]
on localizations, whose left and right adjoints are respectively obtained by applying the localization functor $\RelCati \xra{\locL} \Cati$ to the composites
\[ \M^c \hookra \M \xra{F} \N \]
and
\[ \M \xla{G} \N \hookla \N^f . \]
\end{thm}

\begin{defn}\label{defn der functors of q adjn}
Given a Quillen adjunction $F \adj G$, we refer to the the resulting adjunction $\bbL F \adj  \bbR G$ on localizations of \cref{adjunction thm} as its \bit{derived adjunction}.  We refer to $\bbL F$ as the \bit{left derived functor} of $F$, and to $\bbR G$ as the \bit{right derived functor} of $G$.
\end{defn}

\cref{adjunction thm} has the following easy consequence.

\begin{cor}\label{cor quillen equivce}
The derived adjunction
\[ \bbL F : \loc{\M}{\bW_\M} \adjarr \loc{\N}{\bW_\N} : \bbR G \]
of a Quillen equivalence
\[ F : \M \adjarr \N : G \]
of model $\infty$-categories is an adjoint equivalence.
\end{cor}

\begin{rem}
With \cref{adjunction thm} in hand, to prove \cref{cor quillen equivce} it would suffice to show that either one of the derived adjoint functors is an equivalence; this can be accomplished using the \textit{fundamental theorem of model $\infty$-categories} (\Cref{fundthm:fundamental theorem}), which provides an explicit description of the hom-spaces in the localization of a model $\infty$-category.  However, our proofs of \cref{adjunction thm} and of \cref{cor quillen equivce} will not rely on that result (recall \cref{intro rmk subcats of good objs}).
\end{rem}

\begin{rem}
A number of examples of Quillen adjunctions and Quillen equivalences are provided in \cref{sspaces:subsection examples of adjunctions}.
\end{rem}

\subsection{Homotopy co/limits}\label{subsection ho-co-lims}

Some of the most important operations one can perform within an $\infty$-category are the extraction of \textit{limits} and \textit{colimits}.  However, co/limit functors on relative $\infty$-categories do not generally take natural weak equivalences to weak equivalences.  In view of the theory of derived adjunctions laid out in \cref{subsection qadjns and qeqs}, in the setting of model $\infty$-categories it is therefore important to determine sufficient conditions under which co/limit functors can be derived, i.e.\! under which they determine left/right Quillen functors.

We now codify this desired situation.

\begin{notn}
For a model $\infty$-category $\M$ and an $\infty$-category $\C$, we write $\bW_{\Fun(\C,\M)} \subset \Fun(\C,\M)$ for the subcategory of natural weak equivalences.  Of course, considering $(\M,\bW)$ as a relative $\infty$-category, via \cref{rnerves:define internal hom in relcats} this identifies as
\[ \begin{tikzcd}
\bW_{\Fun(\C,\M)} \arrow[hook]{r} \arrow{d}[sloped, anchor=north]{\sim} & \Fun(\C,\M) \arrow{d}[sloped, anchor=south]{\sim} \\
\Fun(\min(\C),\M)^\bW \arrow[hook]{r} & \Fun(\min(\C),\M)^\Rel .
\end{tikzcd} \]
\end{notn}

\begin{defn}
Let $\M$ be a model $\infty$-category, and let $\C$ be an $\infty$-category.  Suppose that $\M$ admits $\C$-shaped colimits, so that we obtain an adjunction
\[ \colim : \Fun(\C,\M) \adjarr \M : \const . \]
If $(\Fun(\C,\M),\bW_{\Fun(\C,\M)})$ admits a model structure such that this adjunction becomes a Quillen adjunction, we refer to its resulting left derived functor
\[ \bbL \! \colim : \loc{\Fun(\C,\M)}{\bW_{\Fun(\C,\M)}} \ra \loc{\M}{\bW_\M} \]
as a \bit{homotopy colimit} functor.  Dually, suppose that $\M$ admits $\C$-shaped limits, so that we obtain an adjunction
\[ \const : \M \adjarr \Fun(\C,\M) : \lim \! . \]
If $(\Fun(\C,\M),\bW_{\Fun(\C,\M)})$ admits a model structure such that this adjunction becomes a Quillen adjunction, we refer to its resulting right derived functor
\[ \loc{\M}{\bW_\M} \la \loc{\Fun(\C,\M)}{\bW_{\Fun(\C,\M)}} : \bbR \! \lim \]
as a \bit{homotopy limit} functor.
\end{defn}

Now, to check that an adjunction between model $\infty$-categories is a Quillen adjunction, it suffices to show only that either its left adjoint is a left Quillen functor or that its right adjoint is a right Quillen functor.  This leads us to define the following ``absolute'' model structures on functor $\infty$-categories.

\begin{defn}\label{defn proj and inj model strs}
Let $\M$ be a model $\infty$-category, and let $\C$ be an $\infty$-category.  Suppose that there exists a model structure on $\Fun(\C,\M)$ whose weak equivalences and fibrations are determined objectwise.  In this case, we call this as the \bit{projective model structure}, and denote it by $\Fun(\C,\M)_\projective$.  Dually, suppose that there exists a model structure on $\Fun(\C,\M)$ whose weak equivalences and cofibrations are determined objectwise.  In this case, we call this the \bit{injective model structure}, and denote it by $\Fun(\C,\M)_\injective$.
\end{defn}

\begin{rem}\label{rem q adjns for proj and inj model strs}
\cref{defn proj and inj model strs} immediately implies
\begin{itemizesmall}
\item
that whenever $\M$ admits $\C$-shaped colimits and there exists a projective model structure on $\Fun(\C,\M)$, then we obtain a Quillen adjunction
\[ \colim : \Fun(\C,\M)_\projective \adjarr \M : \const , \]
and
\item that that whenever $\M$ admits $\C$-shaped limits and there exists an injective model structure on $\Fun(\C,\M)$, then we obtain a Quillen adjunction
\[ \const : \M \adjarr \Fun(\C,\M)_\injective : \lim \! . \]
\end{itemizesmall}
\end{rem}

\begin{rem}\label{rem proj needs cofgen while inj needs comb}
As in the classical case, the projective and injective model structures do not always exist.  However, it appears that
\begin{itemizesmall}
\item the projective model structure will exist whenever $\M$ is cofibrantly generated (see \cref{sspaces:section cofgen model infty-cats}), while
\item the injective model structure will exist whenever $\M$ is \textit{combinatorial} (that is, its underlying $\infty$-category is presentable and its model structure is cofibrantly generated);
\end{itemizesmall}
see \cite[Theorem 11.6.1]{Hirsch} and Proposition T.A.2.8.2.\footnote{In the construction of the projective model structure, one can replace the appeal to the ``set of objects'' of $\C$ with an arbitrary surjective map $C \ra \C$ from some $C \in \Set \subset \Cati$; the necessary left Kan extension will exist as long as $\M$ is cocomplete, which seems to be generally true in practice.  However, there is also some subtlety regarding whether the resulting sets of would-be generating cofibrations and generating acyclic cofibrations do indeed admit the small object argument: it suffices that the set $I$ (resp.\! $J$) of generating (resp.\! acyclic) cofibrations have that all the sources of its elements be small with respect to the \textit{tensors} of its elements over the various hom-spaces of $\C$.  However, it similarly seems that in practice these objects will in fact be small with respect to the entire $\infty$-category $\M$, so that this is not actually an issue.}
\end{rem}

\subsection{Reedy model structures}\label{subsection reedy model str}

Whereas the projective and injective model structures of \cref{defn proj and inj model strs} are not always known to exist (even for model 1-categories), there is a class of examples in which a model structure on $(\Fun(\C,\M),\bW_{\Fun(\C,\M)})$ is always guaranteed to exist: the \textit{Reedy model structure}.  This does not make any additional assumptions on the model $\infty$-category $\M$ (recall \cref{rem proj needs cofgen while inj needs comb}), but instead it requires that $\C$ be a (strict) 1-category equipped with a certain additional structure.

The Reedy model structure will be useful in a number of settings: we'll use \cref{ex reedy structure on n} a number of times in the proof of \cref{adjunction thm}, it will be heavily involved in our development of ``cylinder objects'' and ``path objects'' in model $\infty$-categories in \cref{fundthm:section fundamental theorem} (leading towards the fundamental theorem of model $\infty$-categories (\Cref{fundthm:fundamental theorem})), and it is also closely related to the resolution model structure (see e.g.\! \cref{sspaces:subsection GHOsT motivation}).


We begin by fixing the following definition. 

\begin{defn}\label{define reedy cat}
Let $\C \in \strcat$ be a gaunt category equipped with a factorization system defined by two wide subcategories $\rvec{\C} , \lvec{\C} \subset \C$; that is, every morphism $\varphi$ in $\C$ admits a unique factorization as a composite $\rvec{\varphi} \circ \lvec{\varphi}$, where $\rvec{\varphi}$ is in $\rvec{\C}$ and $\lvec{\varphi}$ is in $\lvec{\C}$.  Suppose there do not exist any infinite ``decreasing'' zigzags of non-identity morphisms in $\C$, where by ``decreasing'' we mean that all forward-pointing arrows lie in $\lvec{\C}$ and all backwards-pointing arrows lie in $\rvec{\C}$.  Then, we say that $\C$ is a \bit{Reedy category}, and we refer to the defining subcategories $\rvec{\C},\lvec{\C} \subset \C$ respectively as its \bit{direct subcategory} and its \bit{inverse subcategory}.
\end{defn}

\begin{rem}\label{different defns of reedy cats}
\cref{define reedy cat} is lifted from Definition T.A.2.9.1.  There is also a more restrictive definition in the literature, given for instance as \cite[Definition 15.1.2]{Hirsch}, in which one requires that $\C$ comes equipped with a ``degree function'' $\deg : \Nerve(\C)_0 \ra \bbN$ such that all non-identity morphisms in $\rvec{\C}$ raise degree while all non-identity morphisms in $\lvec{\C}$ lower degree: the nonexistence of infinite decreasing zigzags then follows from the fact that $\bbN$ has a minimal element.

However, as pointed out in \cite[Remark 15.1.4]{Hirsch}, the results of \cite[Chapter 15]{Hirsch} easily generalize to the case when the degree function takes values in ordinals rather than simply in nonnegative integers.  Indeed, Notation T.A.2.9.11 introduces the notion of a ``good filtration'' on a Reedy category, which is a transfinite total ordering of its objects that effectively serves the same purpose as an ordinary degree function (although note that a degree function need not be injective in general), and Remark T.A.2.9.12 observes that good filtrations always exist.

In any case, these data (either degree functions or good filtrations) both reflect the most important feature of Reedy categories, namely their amenability to inductive manipulations.  In practice, we will generally only use Reedy categories of the more restrictive sort, but it is no extra effort to work in the more general setting.
\end{rem}

\begin{defn}\label{latching and matching cats}
Given a Reedy category $\C$, we define its \bit{latching category} at an object $c \in \C$ to be the full subcategory
\[ \partial \left( \rvec{\C}_{/c} \right) \subset \rvec{\C}_{/c} \]
on all objects besides $\id_c$, and we define its \bit{matching category} at an object $c \in \C$ to be the full subcategory
\[ \partial \left( \lvec{\C}_{c/} \right) \subset \lvec{\C}_{c/} \]
on all objects besides $\id_c$.
\end{defn}

\begin{rem}\label{notation latching and matching objects}
We will assume familiarity with the basic theory of Reedy categories.  For further details, we refer the reader to \cite[Chapter 15]{Hirsch} or to \sec T.A.2.9 (with the caveat that the latter source works in somewhat greater generality than the former, as explained in \cref{different defns of reedy cats}).  In particular, given a functor
\[ \partial \left( \rvec{\C}_{/c} \right) \xra{F} \M \]
(e.g.\! the restriction of a functor $\C \xra{F} \M$) we will write
\[ \Latch_c(F) = \colim_{\partial \left( \rvec{\C}_{/c} \right)}(F) , \]
and given a functor
\[ \partial \left( \lvec{\C}_{c/} \right) \xra{G} \M \]
(e.g.\! the restriction of a functor $\C \xra{G} \M$) we will write
\[ \Match_c(F) = \lim_{\partial \left( \lvec{\C}_{c/} \right)}(G) . \]
(This notation jibes with that of item \Cref{sspaces:section conventions}\cref{sspaces:generalized matching and latching}).
\end{rem}

\begin{rem}\label{reedy constrns work in infty-cats}
In general, the usual constructions with Reedy categories go through equally well when the target is an $\infty$-category.  In particular, we explicitly record here that given a bicomplete $\infty$-category $\M$ and a Reedy category $\C$, one can inductively construct both objects and morphisms of $\Fun(\C,\M)$ in exactly the same manner as when $\M$ is merely a category, using latching/matching objects and (relative) latching/matching maps.  For the construction of objects this is observed as Remark T.A.2.9.16, but both of these assertions follow easily from Proposition T.A.2.9.14.
\end{rem}

As indicated at the beginning of this subsection, the primary reason for our interest in Reedy categories is the following result.

\begin{thm}\label{reedy model str thm}
Let $\M$ be a model $\infty$-category, and let $\C$ be a Reedy category.  Then there exists a model structure on $\Fun(\C,\M)$, in which a map $F \ra G$ is
\begin{itemize}
\item a weak equivalence if and only if the induced maps
\[ F(c) \ra G(c) \]
are in $\bW \subset \M$ for all $c \in \C$,
\item a (resp.\! acyclic) cofibration if and only if the relative latching maps
\[ F(c) \coprod_{\Latch_c(F)} \Latch_c(G) \ra G(c) \]
are in $\bC \subset \M$ (resp.\! $\bW \cap \bC \subset \M$) for all $c \in \C$,
\item a (resp.\! acyclic) fibration if and only if the relative matching maps
\[ F(c) \ra \Match_c(F) \underset{\Match_c(G)}{\times} G(c) \]
are in $\bF \subset \M$ (resp.\! $\bW \cap \bF \subset \M$) for all $c \in \C$.
\end{itemize}
\end{thm}

\begin{proof}
The proof is identical to that of Proposition T.A.2.9.19 (or to those of \cite[Theorems 15.3.4(1) and 15.3.5]{Hirsch}).
\end{proof}

\begin{defn}\label{define reedy model str}
We refer to the model structure of \cref{reedy model str thm} as the \bit{Reedy model structure} on $\Fun(\C,\M)$, and we denote this model $\infty$-category by $\Fun(\C,\M)_\Reedy$.
\end{defn}

\begin{ex}\label{ex reedy structure on n}
There is a Reedy category structure on $[n] \in \bD \subset \strcat$ determined by the degree function $\deg(i) = i$.  As the inverse subcategory $\lvec{[n]} \subset [n]$ associated to this Reedy category structure consists only of identity maps, the resulting Reedy model structure $\Fun([n],\M)_\Reedy$ coincides with the projective model structure $\Fun([n],\M)_\projective$ of \cref{defn proj and inj model strs}.
\end{ex}

\begin{rem}
In particular, \cref{ex reedy structure on n} shows that the projective model structure $\Fun([n],\M)_\projective$ always exists (without any additional assumptions on $\M$).  We will use this fact repeatedly without further comment.
\end{rem}

\begin{rem}\label{rem qeqs between inj proj and reedy}
It follows essentially directly from the definitions that whenever they all exist, the projective, injective, and Reedy model structures assemble into a commutative diagram
\[ \begin{tikzcd}[row sep=2cm]
\Fun(\C,\M)_\projective \horizadjntwocolumns \diagadjn & & \Fun(\C,\M)_\injective \\
& \Fun(\C,\M)_\Reedy \antidiagadjn
\end{tikzcd} \]
of Quillen equivalences.  (If only two of them exist, then they still participate in the indicated Quillen equivalence.)
\end{rem}

The Reedy model structure is also functorial in exactly the way one would hope.

\begin{thm}
For any Reedy category $\C$, if $\M \adjarr \N$ is a Quillen adjunction (resp.\! Quillen equivalence) of model $\infty$-categories, then the induced adjunction
\[ \Fun(\C,\M)_\Reedy \adjarr \Fun(\C,\N)_\Reedy \]
is a Quillen adjunction (resp.\! Quillen equivalence) as well.
\end{thm}

\begin{proof}
The proof is identical to that of \cite[Proposition 15.4.1]{Hirsch}.
\end{proof}

Of course, much of our interest in functor $\infty$-categories stems from the fact that these are the source of co/limit functors.
Thus, we will often want to know when a co/limit functor is a Quillen functor with respect to a given Reedy model structure.
This will not always be the case.  However, there does exist a class of ``absolute'' examples, as encoded by the following.

\begin{defn}\label{defn co/fibt constants}
Let $\C$ be a Reedy category.  We say that $\C$ has \bit{cofibrant constants} if for every model $\infty$-category $\M$ admitting $\C$-shaped limits, the adjunction
\[ \const : \M \adjarr \Fun(\C,\M)_\Reedy : \lim \]
is a Quillen adjunction.  Dually, we say that $\C$ has \bit{fibrant constants} if for every model $\infty$-category $\M$ admitting $\C$-shaped colimits, the adjunction
\[ \colim : \Fun(\C,\M)_\Reedy \adjarr \M : \const \]
is a Quillen adjunction.
\end{defn}

Such Reedy categories admit the following characterization, which implies that \cref{defn co/fibt constants} actually coincides with the classical definition (see \cite[Definition 15.10.1 and Proposition 15.10.2]{Hirsch}).

\begin{prop}
Let $\C$ be a Reedy category.  Then $\C$ has cofibrant constants if and only if for every $c \in \C$ the latching category $\partial \left( \rvec{\C}_{/c} \right)$ is either empty or connected.  Dually, $\C$ has fibrant constants if and only if for every $c \in \C$ the matching category $\partial \left( \lvec{\C}_{c/} \right)$ is either empty or connected.
\end{prop}

\begin{proof}
If $\C$ has cofibrant constants, then in particular for every model 1-category $\M$ (or really, every model 1-category under the more restrictive \cite[Definition 7.1.3]{Hirsch}) the adjunction $\const : \M \adjarr \Fun(\C,\M)_\Reedy : \lim$ is a Quillen adjunction.  Hence, the fact that the latching category $\partial \left( \rvec{\C}_{/c} \right)$ is either empty or connected for every $c \in \C$ follows from \cite[Theorem 15.10.8(1) and Proposition 15.10.2(1)]{Hirsch}.

Conversely, suppose that for every $c \in \C$ the latching category $\partial \left( \rvec{\C}_{/c} \right)$ is either empty or connected, and suppose that $\M$ is a model $\infty$-category admitting $\C$-shaped limits.  Then for any fixed choice of object $c \in \C$, for any object $z \in \M$ the latching object
\[ \Latch_c(\const(z)) = \colim_{\partial\left(\rvec{\C}_{/c} \right)} \const(z) \]
is either
\begin{itemizesmall}
\item always equivalent to $\es_\M$, or
\item always equivalent to $z$ itself.
\end{itemizesmall}
Hence, for any map $x \ra y$ in $\M$, the relative latching map
\[ (\const(x))(c) \coprod_{\Latch_c (\const(x))} \Latch_c (\const(y)) \ra (\const(y))(c) \]
is either $x \ra y$ or $y \ra y$.  It follows that $\const : \M \ra \Fun(\C,\M)_\Reedy$ is a left Quillen functor, so that the adjunction $\const : \M \adjarr \Fun(\C,\M)_\Reedy : \lim$ is a Quillen adjunction, as desired.

Of course, the dual claim follows from a dual argument.
\end{proof}

\begin{ex}\label{ex cat of simplices has fibt consts}
For any simplicial set $K \in s\Set$, its ``category of simplices'' (i.e.\! the category
\[ \bD_{/K} = \bD \underset{s\Set}{\times} s\Set_{/K} , \]
or equivalently the category $\Gr^-(K) \in \CartFib(\bD)$ obtained by considering $K \in \Fun(\bD^{op},\Set)$) is a Reedy category with fibrant constants; this is proved as \cite[Example 15.1.19 and Proposition 15.10.4]{Hirsch}.  In particular, the category $\Gr^-(\pt_{s\Set}) \cong \bD$ itself has fibrant constants.  By dualizing, we obtain that the category $\bD^{op}$ has cofibrant constants.  \end{ex}

\begin{rem}
Note that in general, the observations of \cref{ex cat of simplices has fibt consts} only provides Quillen adjunctions
\[ \const : \M \adjarr s\M_\Reedy : \lim \]
and
\[ \colim : c\M_\Reedy \adjarr \M : \const , \]
which are rather useless in practice (since $\bD^{op}$ has an initial object and $\bD$ has a terminal object).  To obtain a left Quillen functor $s\M_\Reedy \ra \M$, we will generally need to take a \textit{resolution} of the object $\const(\pt_\M) \in s\M$ (e.g.\! one coming from a \textit{simplicio-spatial model structure} (see \cref{defn sspatial model str})).\footnote{If $\C$ is an $\infty$-category (which is finitely bicomplete and admits geometric realizations) and we equip $\C$ with the \textit{trivial} model structure (see \cref{sspaces:trivial model structure}), then we do obtain a Quillen adjunction
$ |{-}| : s(\C_\triv)_\Reedy \adjarr \C_\triv : \const $.  However, unwinding the definitions, we see that this is really just the Quillen adjunction $|{-}| : (s\C)_\triv \adjarr \C_\triv : \const$.}
\end{rem}

\begin{ex}
The Reedy trick generalizes from model categories to model $\infty$-categories without change.  Recall that the walking span category $\Nerve^{-1}(\Lambda^2_0) = ( \bullet \la \bullet \ra \bullet)$ admits a Reedy category structure determined by the degree function described by the picture $(0 \la 1 \ra 2)$.  Moreover, it is straightforward to verify that this Reedy category has fibrant constants (see e.g.\! the proof of \cite[Proposition 15.10.10]{Hirsch}).  Thus, for any model $\infty$-category $\M$, we obtain a Quillen adjunction
\[ \colim : \Fun(\Nerve^{-1}(\Lambda^2_0),\M)_\Reedy \adjarr \M : \const , \]
in which the cofibrant objects of $\Fun(\Nerve^{-1}(\Lambda^2_1),\M)_\Reedy$ are precisely the diagrams of the form $x \la y \cofibn z$ for $x,y,z \in \M^c \subset \M$.
\end{ex}

\begin{ex}
Clearly, the poset $(\bbN,\leq)$ admits a Reedy structure (defined by the identity map, considered as a degree function) which has fibrant constants.  Thus, for any model $\infty$-category $\M$, we obtain a Quillen adjunction
\[ \colim : \Fun((\bbN,\leq),\M)_\Reedy \adjarr \M : \const , \]
in which the cofibrant objects of $\Fun((\bbN,\leq),\M)_\Reedy$ are precisely those diagrams consisting of cofibrations between cofibrant objects.\footnote{In fact, this Reedy poset has cofibrant constants as well.  However, the resulting Quillen adjunction will be trivial since this poset has an initial object.}  Dually, we obtain a Quillen adjunction
\[ \const : \M \adjarr \Fun((\bbN,\leq)^{op} , \M)_\Reedy : \lim , \]
in which the fibrant objects of $\Fun((\bbN,\leq)^{op} , \M)_\Reedy$ are precisely those diagrams consisting of fibrations between fibrant objects.
\end{ex}

We end this section by recording the following result.

\begin{lem}\label{latching and matching are reedy}
Let $\C$ be a Reedy category, and let $c \in \C$.
\begin{enumerate}
\item\label{latching stuff in latching and matching are reedy}
\begin{enumeratesub}
\item\label{latching cat is reedy}
The latching category $\partial \left( \rvec{\C}_{/c} \right)$ admits a Reedy structure with fibrant constants, in which the direct subcategory is the entire category and the inverse subcategory contains only the identity maps.
\item\label{latching in latching is latching}
With respect to the Reedy structure of part \ref{latching cat is reedy}, the canonical functor $\partial \left(\rvec{\C}_{/c} \right) \ra \C$ induces an isomorphism
\[ \partial \left( \rvec{\partial\left(\rvec{\C}_{/c}\right)}_{/(d \ra c)} \right) \xra{\cong} \partial \left( \rvec{\C}_{/d} \right) \]
of latching categories (from that of the object $(d \ra c) \in \partial\left(\rvec{\C}_{/c}\right)$ to that of the object $d \in \C$).
\end{enumeratesub}
\item\label{matching stuff in latching and matching are reedy}
\begin{enumeratesub}
\item\label{matching cat is reedy}
The matching category $\partial \left( \lvec{\C}_{c/} \right)$ admits a Reedy structure with cofibrant constants, in which the direct subcategory contains only the identity maps and the inverse subcategory is the entire category.
\item\label{matching in matching is matching}
With respect to the Reedy structure of part \ref{matching cat is reedy}, the canonical functor $\partial \left( \lvec{\C}_{c/} \right) \ra \C$ induces an isomorphism
\[ \partial \left( \lvec{\partial\left(\lvec{\C}_{c/}\right)}_{(c \ra d)/} \right) \xra{\cong} \partial \left( \lvec{\C}_{d/}\right) \]
of matching categories (from that of $(c \ra d) \in \partial \left( \lvec{\C}_{c/}\right)$ to that of $d \in \C$).
\end{enumeratesub}
\end{enumerate}
\end{lem}

\begin{proof}
Parts \ref{latching stuff in latching and matching are reedy}\ref{latching cat is reedy} \and \ref{matching stuff in latching and matching are reedy}\ref{matching cat is reedy} are precisely \cite[Proposition 15.10.6]{Hirsch}, and parts \ref{latching stuff in latching and matching are reedy}\ref{latching in latching is latching} \and \ref{matching stuff in latching and matching are reedy}\ref{matching in matching is matching} follow by inspection.
\end{proof}

\section{Relative co/cartesian fibrations and bicartesian fibrations}\label{section rel co-cart-fibns and bicart fibns}

In this brief section, we describe two enhancements of the theory of co/cartesian fibrations which we will need: in \cref{subsection rel co/cart fibns} we study \textit{relative co/cartesian fibrations}, while in \cref{subsection bicart fibns} we study \textit{bicartesian fibrations}.

\subsection{Relative co/cartesian fibrations}\label{subsection rel co/cart fibns}

Suppose we are given a diagram
\[ \begin{tikzcd}
\C \arrow{r}{F} & \RelCati \arrow{r}{\locL} \arrow{d}[swap]{\forget_\Rel} & \Cati \\
& \Cati .
\end{tikzcd} \]
In our proof of \cref{adjunction thm}, we will be interested in the relationship between the upper composite (of the componentwise localization of the diagram $F$ of relative $\infty$-categories) and the cocartesian fibration
\[ \Gr(\forget_\Rel \circ F) \ra \C . \]
In other words, we would like to take some sort of ``fiberwise localization'' of this cocartesian fibration.  In order to do this, we must keep track of the morphisms which we would like to invert.  This leads us to the following terminology.

\begin{defn}\label{define relative co/cart fibns}
Let $\C \in \Cati$, and suppose we are given a commutative diagram
\[ \begin{tikzcd}[column sep=1.5cm]
& \RelCati \arrow{d}{\forget_\Rel} \\
\C \arrow{ru}{F} \arrow{r}[swap]{\forget_\Rel \circ F} & \Cati .
\end{tikzcd} \]
Then, we write $\Gr_\Rel(F)$ for the relative $\infty$-category obtained by equipping $\Gr(\forget_\Rel \circ F)$ with the weak equivalences coming from the lift $F$ of $\forget_\Rel \circ F$.  Note that its weak equivalences all map to equivalences in $\C$, so that we can consider the canonical projection as a map $\Gr_\Rel(F) \ra \min(\C)$ of relative $\infty$-categories.  We write $\coCartFib_\Rel(\C)$ for the $\infty$-category of cocartesian fibrations over $\C$ equipped with such a relative $\infty$-category structure, and we call this the $\infty$-category of \bit{relative cocartesian fibrations} over $\C$.  The Grothendieck construction clearly lifts to an equivalence
\[ \Fun(\C,\RelCati) \xra[\sim]{\Gr_\Rel} \coCartFib_\Rel(\C) . \]
Of course, we have a dual notion of \bit{relative cartesian fibrations} over $\C$; these assemble into an $\infty$-category $\CartFib_\Rel(\C)$, which comes with an equivalence
\[ \Fun(\C^{op},\RelCati) \xra[\sim]{\Grop_\Rel} \CartFib_\Rel(\C) . \]
\end{defn}

\begin{rem}
Note that an arbitrary cocartesian fibration over $\C$ equipped with a subcategory of weak equivalences which project to equivalences in $\C$ does not necessarily define a relative cocartesian fibration: it must be classified by a diagram of relative $\infty$-categories and relative functors between them (i.e.\! the cocartesian edges must intertwine the weak equivalences).  A dual observation holds for cartesian fibrations.
\end{rem}

We can now precisely state and prove our desired correspondence.

\begin{prop}\label{fiberwise localization}
Let $\C \in \Cati$, and let $\C \xra{F} \RelCati$ classify $\Gr_\Rel(F) \in \coCartFib_\Rel(\C)$.  Then the induced map
\[ \locL(\Gr_\Rel(F)) \ra \C \]
is again a cocartesian fibration.  Moreover, we have a canonical equivalence
\[ \locL(\Gr_\Rel(F)) \simeq \Gr(\locL \circ F) \]
in $\coCartFib(\C)$, i.e.\! this cocartesian fibration classifies the composite
\[ \C \xra{F} \RelCati \xra{\locL} \Cati . \]
\end{prop}

\begin{proof}
By \cite[Theorem 7.4]{GHN}, we have a canonical equivalence
\[ \Gr(\locL \circ F) \simeq \colim \left( \TwAr(\C) \ra \C^{op} \times \C \xra{ \C_{-/} \times (\locL \circ F)} \Cati \times \Cati \xra{- \times -} \Cati  \right) . \]
Since the composite $\Cati \xra{\min} \RelCati \xra{\locL} \Cati$ is canonically equivalent to $\id_\Cati$ and the functor $\RelCati \xra{\locL} \Cati$ commutes with finite products by \cref{rnerves:localization preserves finite products}, this can be rewritten as
\[ \Gr(\locL \circ F) \simeq \colim \left( \TwAr(\C) \ra \C^{op} \times \C \xra{(\min \circ \C_{-/}) \times F} \RelCati \times \RelCati \xra{- \times -} \RelCati \xra{\locL} \Cati \right) . \]
Moreover, the functor $\RelCati \xra{\locL} \Cati$ commutes with colimits (being a left adjoint), and so this can be rewritten further as
\[ \Gr(\locL \circ F) \simeq \locL \left( \colim \left( \TwAr(\C) \ra \C^{op} \times \C \xra{(\min \circ \C_{-/}) \times F} \RelCati \times \RelCati \xra{- \times -} \RelCati \right) \right) . \]
On the other hand, $\RelCati \xra{\forget_\Rel} \Cati$ also commutes with colimits (being a left adjoint as well) and is symmetric monoidal for the respective cartesian symmetric monoidal structures, and so we obtain that
\begin{align*}
& \forget_\Rel \left( \colim \left( \TwAr(\C) \ra \C^{op} \times \C \xra{(\min \circ \C_{-/}) \times F} \RelCati \times \RelCati \xra{- \times -} \RelCati \right) \right) \\
& \simeq \colim \left( \TwAr(\C) \ra \C^{op} \times \C \xra{C_{-/} \times (\forget_\Rel \circ F)} \Cati \times \Cati \xra{- \times -} \Cati \right) \\
& \simeq \Gr(\forget_\Rel \circ F),
\end{align*}
again appealing to \cite[Theorem 7.4]{GHN}.  In other words, the underlying $\infty$-category of the relative $\infty$-category
\[ \colim \left( \TwAr(\C) \ra \C^{op} \times \C \xra{(\min \circ \C_{-/}) \times F} \RelCati \times \RelCati \xra{- \times -} \RelCati \right) \]
is indeed $\Gr(\forget_\Rel \circ F)$; moreover, by definition its subcategory of weak equivalences is inherited from the functor $F$, and hence we have an equivalence
\[ \Gr_\Rel(F) \simeq \colim \left( \TwAr(\C) \ra \C^{op} \times \C \xra{(\min \circ \C_{-/}) \times F} \RelCati \times \RelCati \xra{- \times -} \RelCati \right) \]
in $(\RelCati)_{/\min(\C)}$.\footnote{The structure map for the object on the right comes from its canonical projection to
\[ \min(\TwAr(\C)) \simeq \colim \left( \TwAr(\C) \xra{\const(\pt_\RelCati)} \RelCati \right) \]
followed by the composite projection $\min(\TwAr(\C)) \ra \min(\C^{op} \times \C) \ra \min(\C)$.}  Thus, we have obtained an equivalence
\[ \Gr(\locL \circ F) \simeq \locL(\Gr_\Rel(F)) \]
in $(\Cati)_{/\C}$, which completes the proof of both claims.
\end{proof}


\subsection{Bicartesian fibrations}\label{subsection bicart fibns}

Recall that an \textit{adjunction} can be defined as a map to $[1] \in \Cati$ which is simultaneously a cocartesian fibration and a cartesian fibration.  As we will be interested not just in adjunctions but in \textit{families} of adjunctions (e.g.\! two-variable adjunctions), it will be convenient to introduce the following terminology.

\begin{notn}\label{notation for bicart fibns}
Let $\C$ be an $\infty$-category.  We denote by $\biCartFib(\C)$ the $\infty$-category of \textit{bicartesian fibrations} over $\C$.  This is the underlying $\infty$-category of the bicartesian model structure of Theorem A.4.7.5.10; its objects are those functors to $\C$ which are simultaneously cocartesian fibrations and cartesian fibrations, and its morphisms are maps over $\C$ which are simultaneously morphisms of cocartesian fibrations and morphisms of cartesian fibrations (i.e.\! they preserve both cocartesian morphisms and cartesian morphisms).  We thus have canonical forgetful functors
\[ \coCartFib(\C) \hookla \biCartFib(\C) \hookra \CartFib(\C) , \]
which are both inclusions of (non-full) subcategories, and which both admit left adjoints by Remark A.4.7.5.12.  By Proposition A.4.7.5.17, the composite
\[ \biCartFib(\C) \hookra \coCartFib(\C) \xra[\sim]{\Gr} \Fun(\C,\Cati) \]
identifies $\biCartFib(\C)$ with a certain subcategory of $\Fun(\C,\Cati)$,
\begin{itemizesmall}
\item whose objects are those functors $\C \xra{F} \Cati$ such that for every map $c_1 \ra c_2$ in $\C$, the induced functor $F(c_1) \ra F(c_2)$ is a left adjoint, and
\item whose morphisms are those natural transformations satisfying a certain ``right adjointableness'' condition,
\end{itemizesmall}
and dually for the composite
\[ \biCartFib(\C) \hookra \CartFib(\C) \xra[\sim]{\Grop} \Fun(\C^{op},\Cati) . \]
\end{notn}

\begin{rem}\label{bicart-1 is not Adjn}
Giving an adjunction $\C \adjarr \D$ is equivalent to giving an object of $\biCartFib([1])$ equipped with certain identifications of its fibers, which data can be encoded succinctly as an object of the pullback
\[ \lim \left( \begin{tikzcd}
& \biCartFib([1]) \arrow{d}{(\ev_0,\ev_1)} \\
\pt_\Cati \arrow{r}[swap]{(\C,\D)} & \Cati \times \Cati
\end{tikzcd} \right) \]
in $\Cati$.  In other words, the space of objects of this pullback is (canonically) equivalent to that of the $\infty$-category $\Adjn(\C;\D)$.  However, note that morphisms of bicartesian fibrations are quite different from morphisms in $\Adjn(\C;\D)$: a map from an adjunction $F : \C \adjarr \D : G$ to an adjunction $F' : \C \adjarr \D : G'$ is given
\begin{itemizesmall}
\item in $\Adjn(\C;\D)$, by either a natural transformation $F' \ra F$ or a natural transformation $G \ra G'$, but
\item in $\biCartFib([1])$, a certain sort of \textit{commutative square} in $\Cati$.
\end{itemizesmall}
(So the latter is $(\infty,1)$-categorical, while the former is inherently $(\infty,2)$-categorical.)  In fact, it is not hard to see that the above pullback in $\Cati$ actually defines an $\infty$-groupoid: really, this is just a more elaborate version of the difference between $\Fun(\C,\D)$ and
\[ \lim \left( \begin{tikzcd}
& \coCartFib([1]) \arrow{d}{(\ev_0,\ev_1)} \\
\pt_\Cati \arrow{r}[swap]{(\C,\D)} & \Cati \times \Cati
\end{tikzcd} \right) . \]
\end{rem}

Despite \cref{bicart-1 is not Adjn}, we will have use for the following notation.

\begin{notn}\label{restricted coCart fibns over 1}
For $\C,\D \in \Cati$, we denote by $\coCartFib([1];\C,\D)$ the second pullback in \cref{bicart-1 is not Adjn}.  We will use analogous notation for the various variants of this construction (namely cartesian, relative co/cartesesian, and bicartesian fibrations over $[1]$).  For consistency, we will similarly write
\[ \Cati([1];\C,\D) = \lim \left( \begin{tikzcd}[column sep=1.25cm]
 & (\Cati)_{/[1]} \arrow{d}{(\ev_0,\ev_1)} \\
\pt_\Cati \arrow{r}[swap]{(\C,\D)} & \Cati \times \Cati
\end{tikzcd} \right) . \]
For any $\R_1,\R_2 \in \RelCati$, we also set
\[ \RelCati([1];\R_1,\R_2) = \lim \left( \begin{tikzcd}[column sep=1.25cm]
 & (\RelCati)_{/\min([1])} \arrow{d}{(\ev_0,\ev_1)} \\
\pt_\Cati \arrow{r}[swap]{(\R_1,\R_2)} & \RelCati \times \RelCati
\end{tikzcd} \right) . \]
\end{notn}

\begin{rem}\label{rewrite Nervei of Fun using coCart}
Using \cref{restricted coCart fibns over 1}, note that we can identify
\[ \Nervei( \Fun(\C,\D))_\bullet \simeq \coCartFib([1]; [\bullet] \times \C , \D)^\simeq . \]
This identification (and related ones) will be useful in the proof of \cref{parametrized localizns of Q adjns}.
\end{rem}

\section{The proofs of \cref{adjunction thm} and \cref{cor quillen equivce}}\label{section proofs of qadjns and qeqs}

This section is devoted to proving the results stated in \cref{subsection qadjns and qeqs}, namely
\begin{itemizesmall}
\item \cref{adjunction thm} -- that a Quillen adjunction has a canonical derived adjunction --, and
\item \cref{cor quillen equivce} -- that the derived adjunction of a Quillen equivalence is an adjoint equivalence.
\end{itemizesmall}

We begin with the following key result, the proof of which is based on that of \cite[Lemma 2.4.8]{BHH}.

\begin{lem}\label{cat of fibt replacements is weakly contractible}
Let $\M$ be a model $\infty$-category, and let $x \in \M$.  Then
\[ \left( \bW_{\M_{x/}}^f \right)^\gpd \simeq \pt_\S . \]
\end{lem}

\begin{proof}
By \cite[Lemme d'asph\'ericit\'e]{CisInv}, it suffices to show that for any finite directed set considered as a category $\C \in \strcat$, any functor $\C \ra \bW_{\M_{x/}}^f$ is connected to a constant functor by a zigzag of natural transformations in $\Fun(\C,\bW_{\M_{x/}}^f)$.\footnote{\cite[Lemme d'asph\'ericit\'e]{CisInv} can also be proved invariantly (i.e.\! without reference to quasicategories) by using the theory of complete Segal spaces and replacing Cisinski's appeal to the Quillen equivalence $\sd : s\Set_\KQ \adjarr s\Set_\KQ : \Ex$ and the functor $\Ex^\infty$ to their $\infty$-categorical variants (see \cref{sspaces:subsection Ex-infty}).}  Note that such a functor is equivalent to the data of
\begin{itemizesmall}
\item the composite functor $\C \ra \bW_{\M_{x/}}^f \ra \bW_\M^f$, which we will denote by $\C \xra{F} \bW_\M^f$, along with
\item a natural transformation $\const(x) \ra F$ in $\Fun(\C,\bW_\M)$.
\end{itemizesmall}

We now appeal to Cisinski's theory of \textit{left-derivable categories} introduced in \cite[\sec 1]{CisCatDer} (there called ``cat\'egories d\'erivables \`a gauche''), which immediately generalizes to a theory of \textit{left-derivable $\infty$-categories}: one simply replaces sets with spaces and categories with $\infty$-categories.\footnote{However, the notion of \textit{finite direct categories} (there called ``cat\'egories directes finies'') need not be changed.  Note that such categories are gaunt, so 1-categorical pushouts and pullbacks between them compute their respective $\infty$-categorical counterparts.}  Clearly the model $\infty$-category $\M$ is in particular a left-derivable $\infty$-category.  Hence, considering $\C \xra{F} \bW_\M^f \hookra \M$ as an object of $\Fun(\C,\M)$, by \cite[Proposition 1.29]{CisCatDer} there exists a factorization
\[ F \we F' \fibn \pt_{\Fun(\C,\M)} \simeq \const(\pt_\M) \]
of the terminal map in $\Fun(\C,\M)$, where $F \we F'$ is a componentwise weak equivalence and the map $F' \fibn \pt_{\Fun(\C,\M)}$ is a \textit{boundary fibration} (there called ``une fibration bord\'ee'').  In other words, $F'$ is \textit{fibrant on the boundaries} (there called ``fibrant sur les bords''), and in particular by \cite[Corollaire 1.24]{CisCatDer} it is objectwise fibrant.  Thus, we can consider $F \ra F'$ as a morphism in $\Fun(\C,\bW_\M^f)$, and hence for our main goal it suffices to assume that $F$ itself is fibrant on the boundaries.

Now, our map $\const(x) \ra F$ induces a canonical map $x \ra \lim_\C F$ in $\M$ (where this limit exists because $\M$ is finitely complete), and this map in turn admits a factorization
\[ x \we y \fibn {\lim}_\C F . \]
Moreover, \cite[Proposition 1.18]{CisCatDer} implies that $\lim_\C F \in \M$ is fibrant, and hence $y \in \M$ is fibrant as well.  Further, in the commutative diagram
\[ \begin{tikzcd}
\const(x) \arrow{r}{\approx} \arrow{d}[sloped, anchor=north]{\approx} & F \\
\const(y) \arrow{r} \arrow[dashed]{ru} & \const(\lim_\C F) \arrow[leftarrow, crossing over]{lu} \arrow{u}
\end{tikzcd} \]
in $\Fun(\C,\M)$, the dotted arrow is a componentwise weak equivalence by the two-out-of-three property (applied componentwise). This provides the desired zigzag connecting the object
\[ (\C \xra{F} \bW_\M^f , \const(x) \ra F) \in \Fun(\C,\bW_{\M_{x/}}^f ) \]
to a constant functor, namely the object
\[ (\C \xra{\const(y)} \bW_\M^f , \const(x) \ra \const(y)) \in \Fun(\C,\bW_{\M_{x/}}^f ) , \]
which proves the claim.
\end{proof}

This has the following convenient consequence.

\begin{lem}\label{inclusion of fibrants and weak equivalences into weak equivalences is a weak homotopy equivalence}
For any model $\infty$-category $\M$, the inclusion $\bW^f \hookra \bW$ induces an equivalence under the functor $(-)^\gpd : \Cati \ra \S$.
\end{lem}

\begin{proof}
This functor is final by Theorem A (\Cref{gr:theorem A}) and \cref{cat of fibt replacements is weakly contractible}; note that for an object $x \in \bW$, we have an identification
\[ \bW^f \times_\bW \bW_{x/} \simeq \bW_{\M_{x/}}^f . \]
Hence, the assertion follows from \cref{gr:final functor is an equivalence on groupoid-completions}.
\end{proof}

In turn, this allows us to prove the following pair of results, which we will need in the proof of \cref{adjunction thm}.

\begin{prop}\label{inclusion of fibts induces equivce on rezk nerves}
For any model $\infty$-category $\M$, the inclusion $(\M^f,\bW^f) \hookra (\M,\bW)$ induces an equivalence
\[ \NerveRezki(\M^f,\bW^f) \ra \NerveRezki(\M,\bW) \]
in $s\S$.
\end{prop}

\begin{proof}
We must show that for every $n \geq 0$, the map
\[ \preNerveRezki(\M^f,\bW^f)_n \ra \preNerveRezki(\M,\bW)_n \]
in $\Cati$ becomes an equivalence upon applying $(-)^\gpd : \Cati \ra \S$.  By definition, this is the map
\[ \Fun([n],(\M^f,\bW^f))^\bW \ra \Fun([n],(\M,\bW))^\bW . \]
But this is precisely the inclusion
\[ \bW^f_{\Fun([n],\M)_\projective} \hookra \bW_{\Fun([n],\M)_\projective} , \]
which becomes an equivalence upon groupoid completion by \cref{inclusion of fibrants and weak equivalences into weak equivalences is a weak homotopy equivalence}.
\end{proof}

\begin{cor}\label{inclusion of fibts induces equivce on gpd-compln}
For any model $\infty$-category $\M$, the inclusion $(\M^f,\bW^f) \hookra (\M,\bW)$ is a weak equivalence in $(\RelCati)_\BarKan$, i.e.\! it induces an equivalence
\[ \loc{\M^f}{(\bW^f)} \xra{\sim} \loc{\M}{\bW} \]
in $\Cati$.
\end{cor}

\begin{proof}
This follows from \cref{inclusion of fibts induces equivce on rezk nerves} and the global universal property of the Rezk nerve (\cref{rnerves:Rezk nerve computes localization}).
\end{proof}

We now give one more easy result which we will need in the proof of \cref{adjunction thm}, which we refer to as \bit{Kenny Brown's lemma} (for model $\infty$-categories).

\begin{lem}\label{kenny brown}
Let $\M$ be a model $\infty$-category, and let $(\R,\bW_\R) \in \RelCati$ be a relative $\infty$-category such that $\bW_\R \subset \R$ has the two-out-of-three property.  If $\M \ra \R$ is any functor of underlying $\infty$-categories which takes the subcategory $(\bW \cap \bC)^c_\M \subset \M$ into $\bW_\R \subset \R$, then it also takes the subcategory $\bW^c_\M \subset \M$ into $\bW_\R \subset \C$.
\end{lem}

\begin{proof}
Given any map $x \we y$ in $\bW^c_\M \subset \M$, we can construct a diagram
\[ \begin{tikzcd}
x \arrow{rr}{\approx} \arrow[dashed, tail]{rd}[sloped, swap, pos=0.6]{\approx} & & y \arrow[bend left=40, dashed, tail]{ld}[sloped, pos=0.6]{\approx} \\
& z \arrow[dashed, two heads]{ru}[sloped, pos=0.6]{\approx}
\end{tikzcd} \]
in $\M$, i.e.\! a factorization of the chosen map and a section of the second map which are contained in the various subcategories defining the model structure on $\M$ as indicated, exactly as in \cite[Lemma 7.7.1]{Hirsch} (only omitting the assertion of functoriality).  Hence, our functor $\M \ra \R$ must take our chosen map into $\bW_\R \subset \R$ since this subcategory contains all the equivalences, has the two-out-of-three property, and is closed under composition.  This proves the claim.
\end{proof}

We now turn to this section's primary goal.

\begin{proof}[Proof of \cref{adjunction thm}]
Let $(\M + \N) \ra [1]$ denote the bicartesian fibration corresponding to the underlying adjunction $F \adj G$ of the given Quillen adjunction.  Let us equip this with the subcategory of weak equivalences inherited from $\bW_\M \subset \M$ and $\bW_\N \subset \N$; its structure map can then be considered as a map to $\min([1])$ in $\RelCati$.\footnote{Note that this will not generally make this map into a relative cocartesian fibration or a relative cartesian fibration: left and right Quillen functors are not generally functors of relative $\infty$-categories.}  Let us define full relative subcategories
\[ (\M^c + \N^f),(\M^c+\N),(\M+\N^f) \subset (\M+\N) \]
(which inherit maps to $\min([1])$) by restricting to the cofibrant objects of $\M$ and/or to the fibrant objects of $\N$, as indicated by the notation.  Moreover, let us define the functors $F^c$ and $G^f$ to be the composites
\[ \begin{tikzcd}
\M^c \arrow[hook]{r} \arrow[bend left=45]{rr}{F^c} &
\M
{\arrow[transform canvas={yshift=0.7ex}]{r}{F}[swap, transform canvas={yshift=0.25ex}]{\scriptstyle \bot} \arrow[transform canvas={yshift=-0.7ex}, leftarrow]{r}[swap]{G}}
& \N \arrow[hookleftarrow]{r}
& \N^f . \arrow[bend left=45]{ll}{G^f}
\end{tikzcd} \]
Note that $F^c$ and $G^f$ both preserve weak equivalences by Kenny Brown's lemma (\ref{kenny brown}).  It follows that we have a canonical equivalence
\[ (\M^c + \N) \simeq \Gr_\Rel(F^c) \]
in $\coCartFib_\Rel([1])$ and a canonical equivalence
\[ (\M + \N^f) \simeq \Gr^-_\Rel(G^f) \]
in $\CartFib_\Rel([1])$.  By \cref{fiberwise localization} (and its dual), it follows that
\[ \locL(\M^c + \N) \simeq \Gr(\locL \circ F^c) \]
in $\coCartFib([1])$ and that
\[ \locL(\M + \N^f) \simeq \Gr^-(\locL \circ G^f) \]
in $\CartFib([1])$.

Now, by \cref{inclusions of collages are weak equivalences}, the canonical inclusions induce weak equivalences
\[ (\M^c + \N) \xla{\approx} (\M^c + \N^f) \xra{\approx} (\M + \N^f) \]
in $((\RelCati)_{/\min([1])})_\BarKan$.  Hence, applying $\RelCati \xra{\locL} \Cati$ yields a diagram
\[ \Gr(\locL \circ F^c) \simeq \locL(\M^c+\N) \xla{\sim} \locL(\M^c+\N^f) \xra{\sim} \locL(\M+\N^f) \simeq \Gr^-(\locL \circ G^f) \]
in $(\Cati)_{/[1]}$, so that in particular the map $\locL(\M^c + \N^f) \ra [1]$ is a bicartesian fibration (which as a cocartesian fibration corresponds to $F^c$ while as a cartesian fibration corresponds to $G^f$).  Appealing to \cref{inclusion of fibts induces equivce on gpd-compln} (and its dual), we then obtain a diagram
\[ \begin{tikzcd}
\loc{\M}{\bW_\M} & {\loc{\M^c}{(\bW_\M^c)}} \arrow{l}[swap]{\sim} \arrow[hookrightarrow]{r} \arrow{d} & \locL(\M^c+\N^f) \arrow[hookleftarrow]{r} \arrow{d} & \loc{\N^f}{ (\bW_\N^f)} \arrow{r}{\sim} \arrow{d} & \loc{\N}{\bW_\N} \\
& \{0\} \arrow[hookrightarrow]{r} & {[1]} \arrow[hookleftarrow]{r} & \{1\}
\end{tikzcd} \]
in which the squares are fiber inclusions and which, upon making choices of inverses for the equivalences (the spaces of which are contractible), selects the desired adjunction.
\end{proof}

We now prove a key result which we needed in the proof of \cref{adjunction thm}.

\begin{lem}\label{inclusions of collages are weak equivalences}
The inclusions
\[ (\M^c+\N) \hookla (\M^c+\N^f) \hookra (\M + \N^f) \]
are weak equivalences in $((\RelCati)_{/\min([1])})_\BarKan$.
\end{lem}

\begin{proof}
We will show that the inclusion
\[ (\M^c+\N) \hookla (\M^c+\N^f) \]
is a weak equivalence in $((\RelCati)_{/\min([1])})_\BarKan$; the other weak equivalence follows from a dual argument.  By the global universal property of Rezk nerve (\cref{rnerves:Rezk nerve computes localization}), it suffices to show that applying the functor $\RelCati \xra{\NerveRezki} s\S$ to this map yields an equivalence.  This is equivalent to showing that for every $n \geq 0$, the map
\[ \preNerveRezki(\M^c+\N^f)_n \ra \preNerveRezki(\M^c+\N)_n \]
in $\Cati$ becomes an equivalence upon groupoid completion.  By definition, this is the postcomposition map
\[ \Fun([n],(\M^c+\N^f))^\bW \ra \Fun([n],(\M^c+\N))^\bW . \]

Now, observe that since neither $(\M^c + \N^f)$ nor $(\M^c+\N)$ has any weak equivalences covering the unique non-identity map of $[1]$, then these $\infty$-categories decompose as coproducts (in $\Cati$) over the set of possible composite maps $[n] \ra (\M^c + \N^{(f)}) \ra [1]$, and moreover the map between them respects these decompositions.  Thus, it suffices to show that for each choice of structure map $[n] \ra [1]$, the resulting map
\[ \Fun_{/[1]}([n],(\M^c+\N^f))^\bW \ra \Fun_{/[1]}([n],(\M^c+\N))^\bW \]
in $\Cati$ becomes an equivalence upon applying $(-)^\gpd : \Cati \ra \S$.

First of all, we obtain an equivalence of fibers over the constant map $[n] \xra{\const(0)} [1]$.  Moreover, over the constant map $[n] \xra{\const(1)} [1]$, the above map reduces to
\[ \preNerveRezki(\N^f,\bW^f_\N)_n \ra \preNerveRezki(\N,\bW_\N)_n , \]
in which situation the result follows from \cref{inclusion of fibts induces equivce on rezk nerves}.  Thus, let us restrict our attention to the intermediate cases, supposing that our structure map $[n] \ra [1]$ is given by $0,\ldots, i \mapsto 0$ and $i+1,\ldots,n \mapsto 1$, where $0 \leq i < n$.  Let us write $j = n-(i+1)$.  Then, we can reidentify these $\infty$-categories as
\[ \C^{c,f} = \Fun_{/[1]}([n],(\M^c+\N^f))^\bW \simeq \lim \left( \begin{tikzcd}
& & \Fun([j],\N^f)^\bW \arrow{d}{\{0\}} \\
& \Fun([1],\N)^\bW \arrow{r}[swap]{\{1\}} \arrow{d}{\{0\}} & \bW_\N \\
\Fun([i],\M^c)^\bW \arrow{r}[swap]{F^c \circ \{i\}} & \bW_\N
\end{tikzcd} \right) \]
and
\begin{align*}
\C^c = \Fun_{/[1]}([n],(\M^c+\N))^\bW & \simeq \lim \left( \begin{tikzcd}[ampersand replacement=\&]
\& \& \Fun([j],\N)^\bW \arrow{d}{\{0\}} \\
\& \Fun([1],\N)^\bW \arrow{r}[swap]{\{1\}} \arrow{d}{\{0\}} \& \bW_\N \\
\Fun([i],\M^c)^\bW \arrow{r}[swap]{F^c \circ \{i\}} \& \bW_\N
\end{tikzcd} \right) \\
& \simeq \lim \left( \begin{tikzcd}[ampersand replacement=\&]
\& \Fun([j+1],\N)^\bW \arrow{d}{\{0\}} \\
\Fun([i],\M^c)^\bW \arrow{r}[swap]{F^c \circ \{i\}} \& \bW_\N
\end{tikzcd} \right)
\end{align*}
(with the evident map between them).  By Theorem A (\Cref{gr:theorem A}) and \cref{gr:final functor is an equivalence on groupoid-completions}, it suffices to show that for any object
\[ x = \left( (m_0 \ra \cdots \ra m_i) , (F(m_i) \ra n_0) , (n_0 \ra \cdots \ra n_j) \right) \in \C^c , \]
the resulting comma $\infty$-category
\[ \D = \C^{c,f} \underset{\C^c} (\C^c)_{x/} \]
has that $\D^\gpd \simeq \pt_\S$.

Let us write $x|_\N = (n_0 \ra \cdots \ra n_j) \in \Fun([j],\N)^\bW$, and using this let us define the $\infty$-category $\E$ via the commutative diagram
\[ \begin{tikzcd}[column sep=2cm]
& & \E \arrow{rd} \arrow{dd} \\
\D \arrow{rru} \arrow{r} \arrow{dd} & \C^{c,f} \arrow[crossing over]{rr} & & \Fun([j],\N^f)^\bW \arrow{dd} \\
& & \left( \Fun([j],\N)^\bW \right)_{(x|_\N)/} \arrow{rd} \\
(\C^c)_{x/} \arrow{rru} \arrow{r} & \C^c \arrow{rr} \arrow{dddd} \arrow[leftarrow, crossing over]{uu} & & \Fun([j],\N)^\bW \arrow{dd}{\{0\}} \\ \\
& & \Fun([1],\N)^\bW \arrow{r}[swap]{\{1\}} \arrow{dd}{\{0\}} & \bW_\N \\ \\
& \Fun([i],\M^c)^\bW \arrow{r}[swap]{F^c \circ \{i\}} & \bW_\N
\end{tikzcd} \]
in which all of $\C^{c,f}$, $\D$, and $\E$ are defined as pullbacks (which is what provides the functor $\D \ra \E$).  By applying \cref{cat of fibt replacements is weakly contractible} to the model $\infty$-category $\Fun([j],\N)_\projective$ and the object $x|_\N \in \Fun([j],\N)$, we obtain that $\E^\gpd \simeq \pt_\S$.  Moreover, unwinding the definitions, we see that the functor $\D \ra \E$ is a right adjoint, with left adjoint given by taking the object
\[ \left( \begin{tikzcd}
n_0 \arrow{r} \arrow{d}{\nu_0}[sloped, anchor=north]{\approx} & \cdots \arrow{r} & n_j \arrow{d}[swap]{\nu_j}[sloped, anchor=south]{\approx} \\
n_0' \arrow{r} & \cdots \arrow{r} & n_j'
\end{tikzcd} \right) \in \E \]
to the object
\[ \left(
\left( \begin{tikzcd}
m_0 \arrow{r} \arrow{d}{\id_{m_0}}[sloped, anchor=north]{\approx} & \cdots \arrow{r} & m_i \arrow{d}[swap]{\id_{m_i}}[sloped, anchor=south]{\approx} \\
m_0 \arrow{r} & \cdots \arrow{r} & m_i
\end{tikzcd} \right)
,
\left( \begin{tikzcd}
F(m_i) \arrow{r} \arrow{d}{\id_{F(m_i)}}[sloped, anchor=north]{\approx} & n_0 \arrow{d}[swap]{\nu_0}[sloped, anchor=south]{\approx} \\
F(m_i) \arrow{r} & n_0'
\end{tikzcd} \right)
,
\left( \begin{tikzcd}
n_0 \arrow{r} \arrow{d}{\nu_0}[sloped, anchor=north]{\approx} & \cdots \arrow{r} & n_j \arrow{d}[swap]{\nu_j}[sloped, anchor=south]{\approx} \\
n_0' \arrow{r} & \cdots \arrow{r} & n_j'
\end{tikzcd} \right)
\right) \in \D \]
and acting in the expected way on morphisms.\footnote{Rather than exhibit all of the necessary coherences, this existence of this adjunction can be deduced via (the dual of) Proposition T.5.2.2.8 from the evident counit transformation.}  Hence, by \cref{rnerves:adjns induce equivces on gpd-complns}, it follows that $\D^\gpd \simeq \pt_\S$ as well.  This proves the claim.
\end{proof}

Building on the proof of \cref{adjunction thm}, we can now prove \cref{cor quillen equivce}.

\begin{proof}[Proof of \cref{cor quillen equivce}]
We will prove that the unit of the derived adjunction
\[ \bbL F : \loc{\M}{\bW_\M} \adjarr \loc{\N}{\bW_\N} : \bbR G \]
is a natural equivalence; that its counit is also a natural equivalence will follow from a dual argument.  For this, choose any $x \in \loc{\M}{\bW_\M}$, and choose any cofibrant representative $\tilde{x} \in \M^c$.  Then by \cref{adjunction thm}, $F(\tilde{x}) \in \N$ represents $(\bbL F)(x) \in \loc{\N}{\bW_\N}$.  Let us choose any fibrant replacement
\[ F(\tilde{x}) \we \bbR (F(\tilde{x})) \fibn \pt_\N \]
in $\N$.  Then, again by \cref{adjunction thm}, $G(\bbR(F(\tilde{x}))) \in \M$ represents $(\bbR G)((\bbL F)(x)) \in \loc{\M}{\bW_\M}$.  Moreover, it follows from the proof of \cref{adjunction thm} that the unit map of $\bbL F \adj \bbR G$ at $x \in \loc{\M}{\bW_\M}$ is represented by the composite map
\[ \tilde{x} \xra{\eta^{F \adj G}_{\tilde{x}}} G(F(\tilde{x})) \ra G(\bbR(F(\tilde{x}))) \]
in $\M$.  As this composite map is adjoint to the original weak equivalence $F(\tilde{x}) \we \bbR(F(\tilde{x}))$ in $\N$, it must itself be a weak equivalence in $\M$ since $F \adj G$ is a Quillen equivalence.  So the unit of the adjunction $\bbL F \adj \bbR G$ is indeed a natural equivalence.
\end{proof}

\section{Two-variable Quillen adjunctions}\label{section two-var qadjns}

Recall that a model $\infty$-category $\M$ may be thought of as a \textit{presentation} of its localization $\loc{\M}{\bW}$.  The foremost results of this paper -- \cref{adjunction thm} and \cref{cor quillen equivce} -- assert that certain structures on model $\infty$-categories (namely, Quillen adjunctions and Quillen equivalences) descend to corresponding structures on their localizations (namely, derived adjunctions and derived adjoint equivalences).  In this section, we elaborate further on this theme: we define \textit{two-variable Quillen adjunctions} (see \cref{defn two-var Q adjn}), and prove that they induce canonical \textit{derived two-variable adjunctions} (see \cref{two-var adjunction thm}).  For a more leisurely discussion of two-variable Quillen adjunctions (between model 1-categories), we refer the reader to \cite[\sec 4.2]{HoveyModelCats}.

We begin with a few auxiliary definitions.

\begin{defn}\label{defn various products of maps}
Suppose that we are given three $\infty$-categories $\C$, $\D$, and $\E$, along with a two-variable adjunction
\[ \twovaradjCDE \]
between them.
\begin{itemize}
\item We define the corresponding \bit{pushout product} bifunctor
\[ \Fun([1],\C) \times \Fun([1],\D) \xra{- \square -} \Fun([1],\E) \]
to be given by
\[ (c_1 \ra c_2) \square (d_1 \ra d_2) = \left( (c_2 \otimes d_1) \coprod_{(c_1 \otimes d_1)} (c_1 \otimes d_2) \ra d_1 \otimes d_2 \right) . \]
\item We define the corresponding \bit{left pullback product} bifunctor
\[ \Fun([1],\C)^{op} \times \Fun([1],\E) \xra{\enrhom_l^\square(-,-)} \Fun([1],\D) \]
to be given by
\[ \enrhom_l^\square((c_1\ra c_2)^\opobj ,e_1 \ra e_2) = \left(  \enrhom_l(c_2,e_1) \ra \enrhom_l(c_2,e_1) \underset{\enrhom_l(c_1,e_2)}{\times} \enrhom_l(c_1,e_1) \right) . \]
\item We define the corresponding \bit{right pullback product} bifunctor
\[ \Fun([1],\D)^{op} \times \Fun([1],\E) \xra{\enrhom_r^\square(-,-)} \Fun([1],\C) \]
to be given by
\[ \enrhom_r^\square((d_1\ra d_2)^\opobj ,e_1 \ra e_2) = \left(  \enrhom_r(d_2,e_1) \ra \enrhom_r(d_2,e_1) \underset{\enrhom_r(d_1,e_2)}{\times} \enrhom_r(d_1,e_1) \right) . \]
\end{itemize}
\end{defn}

\begin{rem}\label{rem identify maps as pushout products}
In the situation of \cref{defn various products of maps}, the bifunctor $\C \times \D \xra{- \otimes -} \E$ is a left adjoint and hence commutes with colimits.  Thus, we obtain canonical equivalences $\es_\C \otimes \D \simeq \es_\E \simeq c \otimes \es_\D$ for any $c \in \C$ and any $d \in \D$.  It follows that we obtain identifications
\[ (c_1 \ra c_2) \square (\es_\D \ra d) \simeq (c_1 \ra c_2) \otimes d . \]
and
\[ (\es_\C \ra c) \square (d_1 \ra d_2) \simeq c \otimes (d_1 \ra d_2) . \]
Similarly, we obtain an identification
\[ (\es_\C \ra c) \square (\es_\D \ra d) \simeq (\es_\E \ra c \otimes d) . \]
\end{rem}

We can now given the main definition of this subsection.

\begin{defn}\label{defn two-var Q adjn}
Suppose that $\C$, $\D$, and $\E$ are model $\infty$-categories, and suppose we are given a two-variable adjunction
\[ \twovaradjCDE \]
between their underlying $\infty$-categories.  We say that these data define a \bit{Quillen adjunction of two variables} (or simply a \bit{two-variable Quillen adjunction}) if any of the following equivalent conditions is satisfied:
\begin{itemizesmall}
\item the pushout product bifunctor satisfies
\begin{itemizesmall}
\item $\bC_\C \square \bC_\D \subset \bC_\E$,
\item $(\bW \cap \bC)_\C \square \bC_\D \subset (\bW \cap \bC)_\E$, and
\item $\bC_\C \square (\bW \cap \bC)_\D \subset (\bW \cap \bC)_\E$;
\end{itemizesmall}
\item the left pullback product bifunctor satisfies
\begin{itemizesmall}
\item $\enrhom_l^\square(\bC_\C,\bF_\E) \subset \bF_\D$,
\item $\enrhom_l^\square((\bW \cap \bC)_\C,\bF_\E) \subset (\bW \cap \bF)_\D$, and
\item $\enrhom_l^\square(\bC_\C,(\bW \cap \bF)_\E) \subset (\bW \cap \bF)_\D$;
\end{itemizesmall}
\item the right pullback product bifunctor satisfies
\begin{itemizesmall}
\item $\enrhom_r^\square(\bC_\D,\bF_\E) \subset \bF_\C$,
\item $\enrhom_r^\square((\bW \cap \bC)_\D,\bF_\E) \subset (\bW \cap \bF)_\C$, and
\item $\enrhom_r^\square(\bC_\D,(\bW \cap \bF)_\E) \subset (\bW \cap \bF)_\C$.
\end{itemizesmall}
\end{itemizesmall}
\end{defn}

Before stating the main result of this subsection, we must introduce a parametrized version of \cref{adjunction thm}.

\begin{notn}\label{define infty-cats QAdjn and LQAdjt and RQAdjt}
Let $\M$ and $\N$ be model $\infty$-categories.  We write $\QAdjn(\M;\N) \subset \Adjn(\M;\N)$ for the full subcategory on the Quillen adjunctions, and we write $\LQAdjt(\M,\N) \subset \Fun(\M,\N)$ (resp.\! $\RQAdjt(\N,\M) \subset \Fun(\N,\M)$) for the full subcategory of left (resp.\! right) Quillen functors.\footnote{More precisely, in the latter definitions we might refer only to those functors which \textit{admit} right (resp.\! left) adjoints.  The question of whether the resulting adjunction will be a Quillen adjunction is independent of that choice, however.}  Thus, there are evident equivalences
\[ \LQAdjt(\M,\N)^{op} \xla{\sim} \QAdjn(\M;\N) \xra{\sim} \RQAdjt(\N,\M) . \]
Similarly, for model $\infty$-categories $\C$, $\D$, and $\E$, we write $\QAdjn(\C,\D;\E) \subset \Adjn(\C,\D;\E)$ for the full subcategory on the two-variable Quillen adjunctions.
\end{notn}

\begin{lem}\label{parametrized localizns of Q adjns}
For any model $\infty$-categories $\M$ and $\N$, the construction of \cref{adjunction thm} assembles canonically into a functor
\[ \QAdjn(\M;\N) \ra \Adjn \left( \loc{\M}{\bW_\M};\loc{\N}{\bW_\N} \right) . \]
\end{lem}

We will prove \cref{parametrized localizns of Q adjns} below.  First, we state the main result of this section.

\begin{thm}\label{two-var adjunction thm}
Suppose that $\C$, $\D$, and $\E$ are model $\infty$-categories.  Then, a two-variable Quillen adjunction
\[ \twovaradjCDE \]
induces a canonical two-variable adjunction
\[ \dertwovaradjCDE \]
on localizations, whose constituent bifunctors are respectively obtained by applying the localization functor $\RelCati \xra{\locL} \Cati$ to the composites
\[ \compositescomputedertwovaradjCDE . \]
Moreover, this construction assembles canonically into a functor
\[ \QAdjn(\C,\D;\E) \ra \Adjn(\loc{\C}{\bW_\C} , \loc{\D}{\bW_\D} ; \loc{\E}{\bW_\E} ) . \]
\end{thm}

We will prove \cref{two-var adjunction thm} at the end of this section (after the proof of \cref{parametrized localizns of Q adjns}).

\begin{defn}
Given a two-variable Quillen adjunction, we refer to the resulting two-variable adjunction on localizations of \cref{two-var adjunction thm} as its \bit{derived two-variable adjunction}, and we refer to its constituent bifunctors as the \bit{derived bifunctors} of those of the original two-variable Quillen adjunction.
\end{defn}

\begin{rem}
A two-variable adjunction can be thought of as a special sort of indexed family of adjunctions.\footnote{The ``special'' here refers to the fact that functor $\Adjn(\C,\D;\E) \ra \Fun(\C^{op},\Adjn(\D;\E))$ will not generally be surjective.}  Thus, \cref{parametrized localizns of Q adjns} provides a crucial ingredient for the proof of \cref{two-var adjunction thm}.  As a result, it is essentially no more work to prove the parametrized version of \cref{two-var adjunction thm} than it is to prove the unparametrized version.
\end{rem}

\begin{proof}[Proof of \cref{parametrized localizns of Q adjns}]
Our argument takes place in the diagram in $\Cati$ of \cref{diagram in pf of param localizns}.  Our asserted functor is the middle dotted vertical arrow.
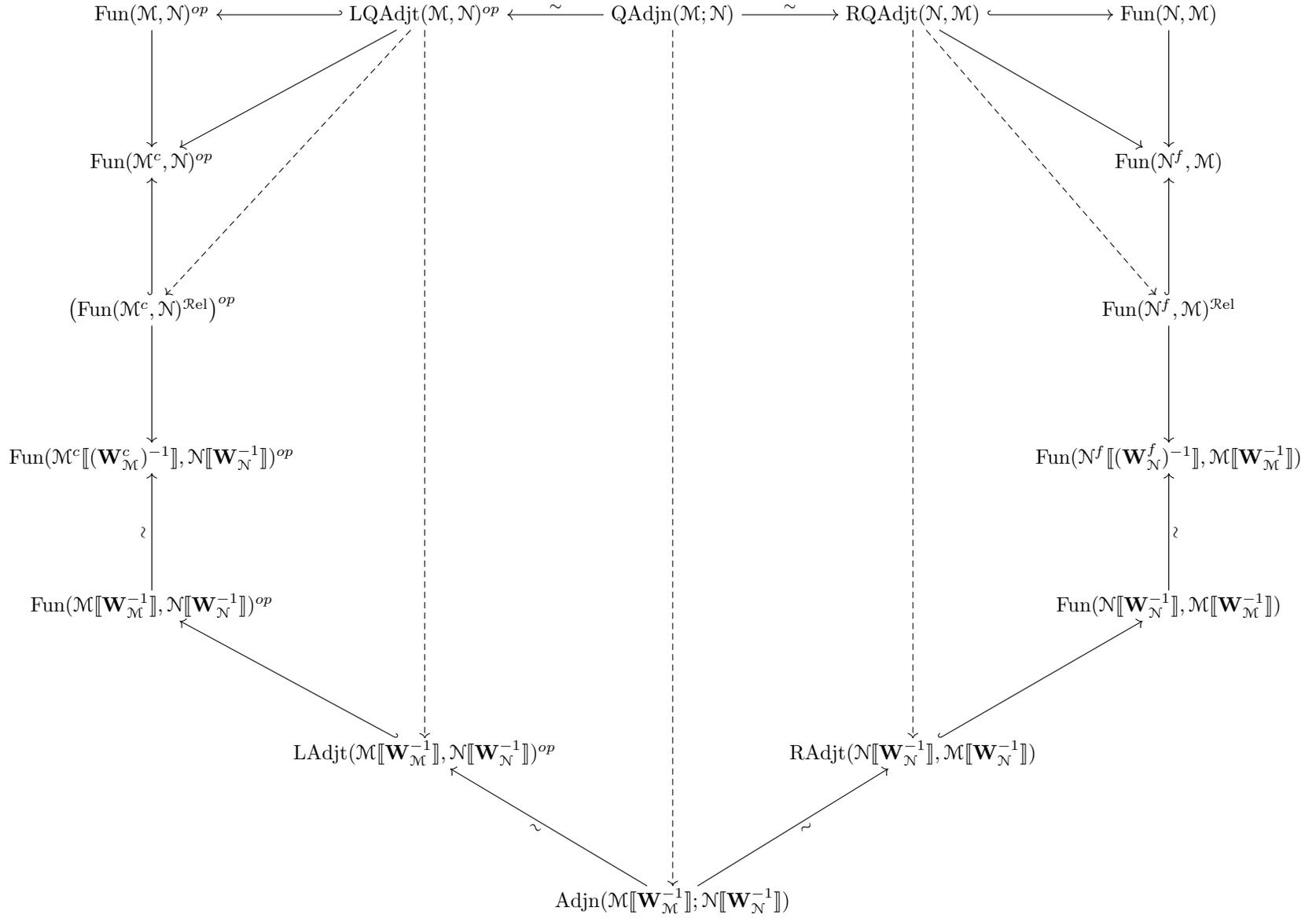
\begin{sidewaysfigure}
\vspace{500pt}
\[ \begin{tikzcd}[ampersand replacement=\&, row sep=2cm, column sep=-0.25cm]
\Fun(\M,\N)^{op} \arrow[hookleftarrow]{r} \arrow{d} \& \LQAdjt(\M,\N)^{op} \arrow{ld} \arrow[dashed]{ldd} \arrow[dashed]{ddddd} \& \QAdjn(\M;\N) \arrow{l}[swap]{\sim} \arrow{r}{\sim} \arrow[dashed]{dddddd} \& \RQAdjt(\N,\M) \arrow[hook]{r} \arrow{rd} \arrow[dashed]{rdd} \arrow[dashed]{ddddd} \& \Fun(\N,\M) \arrow{d} \\
\Fun(\M^c,\N)^{op} \& \& \& \& \Fun(\N^f,\M) \\
\left( \Fun(\M^c,\N)^\Rel \right)^{op} \arrow[hook]{u} \arrow{d} \& \& \& \& \Fun(\N^f,\M)^\Rel \arrow[hook]{u} \arrow{d} \\
\Fun(\loc{\M^c}{(\bW_\M^c)} , \loc{\N}{\bW_\N} )^{op} \& \& \& \& \Fun(\loc{\N^f}{(\bW_\N^f)} , \loc{\M}{\bW_\M} ) \\
\Fun(\loc{\M}{\bW_\M} , \loc{\N}{\bW_\N})^{op} \arrow{u}[sloped, anchor=south]{\sim} \arrow[hookleftarrow]{rd} \& \& \& \& \Fun(\loc{\N}{\bW_\N} , \loc{\M}{\bW_\M}) \arrow{u}[sloped, anchor=north]{\sim} \\
\& \LAdjt(\loc{\M}{\bW_\M} , \loc{\N}{\bW_\N})^{op} \& \& \RAdjt(\loc{\N}{\bW_\N} , \loc{\M}{\bW_\N}) \arrow[hook]{ru} \\
\& \& \Adjn(\loc{\M}{\bW_\M} ; \loc{\N}{\bW_\N}) \arrow{lu}[sloped]{\sim} \arrow{ru}[sloped, swap]{\sim}
\end{tikzcd} \]
\caption{The main diagram in the proof of \cref{parametrized localizns of Q adjns}.}
\label{diagram in pf of param localizns}
\end{sidewaysfigure}
Moreover,
\begin{itemizesmall}
\item the diagonal factorizations follow from Kenny Brown's lemma (\ref{kenny brown}),
\item the vertical maps out of the targets of these factorizations are those of \cref{rnerves:map from rel fctr cat to fctrs betw localizns},
\item the vertical equivalences follow from \cref{inclusion of fibts induces equivce on gpd-compln} (and its dual), and
\item the vertical factorizations follow from \cref{adjunction thm}.
\end{itemizesmall}
Thus, it only remains to show that the diagram commutes, i.e.\! that the two shorter vertical dotted arrows -- which by definition make the outer parts of the diagram commute -- also make the part of the diagram between them commute.

The chief difficulty is in aligning the various sorts of fibrations over $[1]$, which are the setting of the proof of \cref{adjunction thm}, with our $\infty$-categories of adjunctions (recall \cref{bicart-1 is not Adjn}).  We can solve this using \cref{rewrite Nervei of Fun using coCart}.  For instance, via the equivalence $\Nervei : \Cati \xra{\sim} \CSS$, we can identify the right portion of the diagram of \cref{diagram in pf of param localizns} as in \cref{image in CSS of right side of diagram in pf of param localizns}.
\begin{figure}[h]
\[ \begin{tikzcd}[row sep=1.5cm, column sep=-1.5cm]
\Nervei(\RQAdjt(\N,\M))_\bullet \arrow[hook]{r} \arrow{rd} \arrow[dashed]{rdd} \arrow[dashed]{ddddd} & \CartFib([1];\M,[\bullet] \times \N)^\simeq \arrow{d} \\
& \CartFib([1];\M , [\bullet] \times \N^f)^\simeq \\
& \CartFib_\Rel([1];\M , [\bullet] \times \N^f)^\simeq \arrow[hook]{u} \arrow{d} \\
& \CartFib([1];\loc{\M}{\bW_\M},[\bullet] \times \loc{\N^f}{(\bW_\N^f)})^\simeq \\
& \CartFib([1];\loc{\M}{\bW_\M},[\bullet] \times \loc{\N}{\bW_\N})^\simeq \arrow{u}[sloped, anchor=north]{\sim} \\
\Nervei(\RAdjt(\loc{\M}{\bW_\M},\loc{\N}{\bW_\N}))_\bullet \arrow[hook]{ru}
\end{tikzcd} \]
\caption{The nerve of the right portion of the diagram of \cref{diagram in pf of param localizns}.}
\label{image in CSS of right side of diagram in pf of param localizns}
\end{figure}

However, we have not quite reached a symmetric state of affairs: we would like to somehow relate this to the corresponding identifications of the nerves of the left side of the diagram of \cref{diagram in pf of param localizns}, but for instance we have
\[ \Nervei(\Fun(\M,\N))_\bullet \simeq \coCartFib([1];[\bullet] \times \M,\N)^\simeq , \]
and the fibers here do not match up with those in \cref{image in CSS of right side of diagram in pf of param localizns} (nor is this rectified by the fact that we're actually interested in $\Fun(\M,\N)^{op}$ (recall \cref{rnerves:op on CSS})).  To rectify this, we observe that for any $n \geq 0$ and any $\C,\D \in \Cati$, we have a canonical map
\[ \Nervei(\Fun(\C,\D))_n
\simeq \hom_\Cati([n] \times \C, \D)
\ra \hom_\Cati([n] \times \C , [n]) \times \hom_\Cati([n] \times \C , \D)
\simeq \hom_\Cati([n] \times \C , [n] \times \D) \]
selected by the point $\pr_{[n]} \in \hom_\Cati([n] \times \C,[n])$, and this target in turn admits a forgetful map
\[ \hom_\Cati ([n] \times \C , [n] \times \D) \simeq \coCartFib([1];[n] \times \C , [n] \times \D) \ra \Cati([1];[n] \times \C , [n] \times \D ) . \]
Bootstrapping this technique up to the relative case (and piecing the maps together for all objects $[n]^\opobj \in \bD^{op}$), we obtain the diagram of \cref{fatten up fibers over both 0 and 1 in pf of param localizns}, which provides an inclusion of the right edge of the diagram of \cref{image in CSS of right side of diagram in pf of param localizns} into various complete Segal spaces whose constituent spaces now consists of maps to $[1]$ whose fibers over \textit{both} objects $0 \in [1]$ and $1 \in [1]$ are ``fattened up''.

\begin{figure}[h]
\[ \begin{tikzcd}[row sep=1.5cm]
\Cati([1];[\bullet] \times \M,[\bullet] \times \N)^\simeq \arrow[hookleftarrow]{r} \arrow{d} & \CartFib([1];\M,[\bullet] \times \N)^\simeq \arrow{d} \\
\Cati([1];[\bullet] \times \M , [\bullet] \times \N^f)^\simeq \arrow[hookleftarrow]{r} & \CartFib([1];\M , [\bullet] \times \N^f)^\simeq \\
\RelCati([1];[\bullet] \times \M , [\bullet] \times \N^f)^\simeq \arrow[hookleftarrow]{r} \arrow[hook]{u} \arrow{d} & \CartFib_\Rel([1];\M , [\bullet] \times \N^f)^\simeq \arrow[hook]{u} \arrow{d} \\
\Cati([1] ; [\bullet] \times \loc{\M}{\bW_\M},[\bullet] \times \loc{\N^f}{(\bW_\N^f)})^\simeq \arrow[hookleftarrow]{r} & \CartFib([1];\loc{\M}{\bW_\M},[\bullet] \times \loc{\N^f}{(\bW_\N^f)})^\simeq \\
\Cati([1] ; [\bullet] \times \loc{\M}{\bW_\M},[\bullet] \times \loc{\N}{\bW_\N})^\simeq \arrow[hookleftarrow]{r} \arrow{u}[sloped, anchor=south]{\sim} & \CartFib([1];\loc{\M}{\bW_\M},[\bullet] \times \loc{\N}{\bW_\N})^\simeq \arrow{u}[sloped, anchor=north]{\sim}
\end{tikzcd} \]
\caption{An inclusion of the right edge of the diagram of \cref{image in CSS of right side of diagram in pf of param localizns}.}
\label{fatten up fibers over both 0 and 1 in pf of param localizns}
\end{figure}

From here, we only need mimic the proof of \cref{adjunction thm} and restrict further along the inclusion $\M^c \subset \M$: as displayed in the diagram of \cref{Mc and Nf objects in pf of param localizns}, the lower part of the left edge of the diagram of \cref{fatten up fibers over both 0 and 1 in pf of param localizns} admits an inclusion into a map which is now completely self-dual.
\begin{figure}[h]
\[ \begin{tikzcd}[row sep=1.5cm]
\RelCati([1];[\bullet] \times \M^c , [\bullet] \times \N^f)^\simeq \arrow[hookleftarrow]{r} \arrow{d} & \RelCati([1];[\bullet] \times \M , [\bullet] \times \N^f)^\simeq \arrow{d} \\
\Cati([1] ; [\bullet] \times \loc{\M}{\bW_\M},[\bullet] \times \loc{\N^f}{(\bW_\N^f)})^\simeq & \Cati([1] ; [\bullet] \times \loc{\M}{\bW_\M},[\bullet] \times \loc{\N^f}{(\bW_\N^f)})^\simeq \arrow{l}[swap]{\sim} \\
& \Cati([1] ; [\bullet] \times \loc{\M}{\bW_\M},[\bullet] \times \loc{\N}{\bW_\N})^\simeq \arrow{u}[sloped, anchor=north]{\sim} \arrow{lu}[sloped]{\sim}
\end{tikzcd} \]
\caption{The restriction along $\M^c \subset \M$ of the lower part of the left edge of the diagram of \cref{fatten up fibers over both 0 and 1 in pf of param localizns}.}
\label{Mc and Nf objects in pf of param localizns}
\end{figure}
This, finally, gives us a common home for the left and right sides of the diagram of \cref{diagram in pf of param localizns}: its left side
\begin{itemizesmall}
\item admits an identification of its nerve as in \cref{image in CSS of right side of diagram in pf of param localizns}, which in turn
\item admits an inclusion into certain ``fattened up'' objects as in \cref{fatten up fibers over both 0 and 1 in pf of param localizns}, which finally
\item connects, by restricting along the inclusion $\N^f \subset \N$, to the very same map
\[ \begin{tikzcd}[row sep=1.5cm]
\RelCati([1];[\bullet] \times \M^c , [\bullet] \times \N^f)^\simeq \arrow{d} \\
\Cati([1] ; [\bullet] \times \loc{\M}{\bW_\M},[\bullet] \times \loc{\N^f}{(\bW_\N^f)})^\simeq
\end{tikzcd} \]
as that on the left edge in \cref{Mc and Nf objects in pf of param localizns}.
\end{itemizesmall}
It is now simply a matter of unwinding the definitions to see that the middle part of the diagram in \cref{diagram in pf of param localizns} does indeed commute: all the localization functors admit full inclusions into the one indicated just above, and the $\infty$-category
\[ \Adjn(\loc{\M}{\bW_\M};\loc{\N}{\bW_\N}) \]
includes as a full subcategory of its target by, after breaking symmetry, once again appealing to the trick of selecting a canonical projection map to $[n] \in \Cati$ (though the entire point is that the two different ways of obtaining this inclusion are canonically equivalent).  This proves the claim.
\end{proof}

\begin{proof}[Proof of \cref{two-var adjunction thm}]
By \cref{rem identify maps as pushout products}, for any $c \in \C^c$ the induced adjunction
\[ c \otimes - : \D \adjarr \E : \enrhom_l(c,-) \]
is a Quillen adjunction.  Thus, we obtain a factorization
\[ \begin{tikzcd}[row sep=1.5cm]
\Fun(\C^{op} , \Fun (\D^{op} \times \E , \S)) \arrow[hookleftarrow]{r} & \Fun(\C^{op} , \Adjn(\D;\E)) \arrow{d} \\
\Adjn(\C,\D;\E) \arrow[hook]{u} \arrow{ru} \arrow{r} & \Fun((\C^c)^{op} , \Adjn(\D;\E)) \\
\QAdjn(\C,\D;\E) \arrow[hook]{u} \arrow{ru} \arrow[dashed]{r} & \Fun((\C^c)^{op} , \QAdjn(\D;\E)) , \arrow[hook]{u} 
\end{tikzcd} \]
which we compose the functor $(\C^c)^{op} \ra \QAdjn(\D,\E)$ selected by our two-variable Quillen adjunction with the canonical functor of \cref{parametrized localizns of Q adjns} to obtain a composite functor
\[ (\C^c)^{op} \ra \QAdjn(\D;\E) \ra \Adjn(\loc{\D}{\bW_\D} ; \loc{\E}{\bW_\E}) . \]

We claim that this composite functor takes weak equivalences to equivalences.  To see this, suppose first that we are given an acyclic cofibration $c_1 \wcofibn c_2$ in $\C^c$.  Again by \cref{rem identify maps as pushout products}, for any $d \in \D^c$ the induced adjunction
\[ - \otimes d : \C \adjarr \E : \enrhom_r(d,-) \]
is a Quillen adjunction, so that in particular we obtain an acyclic cofibration
\[ c_1 \otimes d \wcofibn c_2 \otimes d \]
is an acyclic cofibration in $\E$.  Since by \cref{adjunction thm} the derived left adjoints of these Quillen adjunctions $\D \adjarr \E$ are computed by localizing the composite $\D^c \hookra \D \ra \E$, it follows that the induced map $(c_1 \otimes - ) \ra (c_2 \otimes -)$ in
\[ \LQAdjt(\D,\E) \simeq \QAdjn(\D;\E)^{op} \]
does indeed descend to an equivalence in
\[ \LAdjt(\loc{\D}{\bW_\D} , \loc{\E}{\bW_\E}) \simeq \Adjn(\loc{\D}{\bW_\D} ; \loc{\E}{\bW_\E})^{op} . \]
The claim now follows from Kenny Brown's lemma (\ref{kenny brown}).  We therefore obtain a factorization
\[ \begin{tikzcd}[row sep=1.5cm]
(\C^c)^{op} \arrow{d} \arrow{r} & \Adjn(\loc{\D}{\bW_\D} ; \loc{\E}{\bW_\E}) \\
( \loc{\C^c}{(\bW_\C^c)} )^{op} \arrow[dashed]{ru}
\end{tikzcd} \]
which, appealing to \cref{rnerves:map from rel fctr cat to fctrs betw localizns}, in fact arises from the induced factorization in the diagram
\[ \begin{tikzcd}[row sep=1.5cm]
\QAdjn(\C,\D;\E) \arrow{r} \arrow[dashed]{rd} & \Fun((\C^c)^{op}, \Adjn(\loc{\D}{\bW_\D} ; \loc{\E}{\bW_\E}) ) \\
& \Fun ( (\C^c)^{op} , \min ( \Adjn(\loc{\D}{\bW_\D} ; \loc{\E}{\bW_\E})) ) ^\Rel \arrow[hook]{u} \arrow{d} \\
& \Fun( \loc{\C^c}{(\bW_\C^c)} , \Adjn(\loc{\D}{\bW_\D} ; \loc{\E}{\bW_\E}) ) .
\end{tikzcd} \]

Thus, it only remains to show that we have a further factorization
\[ \begin{tikzcd}[row sep=1.5cm]
\QAdjn(\C,\D;\E) \arrow{r} \arrow[dashed, bend right=15]{rdd} & \Fun( \loc{\C^c}{(\bW_\C^c)} , \Adjn(\loc{\D}{\bW_\D} ; \loc{\E}{\bW_\E}) ) \\
& \Fun( \loc{\C}{\bW_\C} , \Adjn(\loc{\D}{\bW_\D} ; \loc{\E}{\bW_\E}) ) \arrow{u}[sloped, anchor=north]{\sim} \\
& \Adjn( \loc{\C}{\bW_\C} , \loc{\D}{\bW_\D} ; \loc{\E}{\bW_\E}) \arrow[hook]{u}
\end{tikzcd} \]
which does not depend on our having privileged $\C$ among the model $\infty$-categories $\C$, $\D$, and $\E$ participating in our two-variable Quillen adjunction.  We accomplish these tasks simultaneously by replacing $\C$ with $\D$ in the above arguments: by essentially the same argument as the one given in the proof of \cref{two-var adjunction thm} for why the diagram of \cref{diagram in pf of param localizns} commutes, one sees that we have a commutative square
\[ \begin{tikzcd}[row sep=1.5cm, column sep=-3cm]
& \QAdjn(\C,\D;\E) \arrow{ld} \arrow{rd} \\
\Fun( (\loc{\C}{\bW_\C})^{op} , \Adjn(\loc{\D}{\bW_\D};\loc{\E}{\bW_\E})) \arrow[hook]{rd} & & \Fun( (\loc{\D}{\bW_\D})^{op} , \Adjn(\loc{\C}{\bW_\C};\loc{\E}{\bW_\E}))
\\
& \Fun( (\loc{\C}{\bW_\C})^{op} \times (\loc{\D}{\bW_\D})^{op} \times \loc{\E}{\bW_\E} , \S ) ,\arrow[hookleftarrow]{ru}
\end{tikzcd} \]
which shows
\begin{itemizesmall}
\item that those trifunctors in the image of either of the two (equivalent) composites are indeed co/representable in all variables and hence define two-variable adjunctions, and
\item that the resulting functor
\[ \QAdjn(\C,\D;\E) \ra \Adjn( \loc{\C}{\bW_\C} , \loc{\D}{\bW_\D} ; \loc{\E}{\bW_\E} ) \]
is indeed completely independent of the choice of $\C$, since rotating the two-variable (Quillen) adjunctions involved -- which really just amounts to reordering and passing to opposites as appropriate -- clearly does not affect the induced functor either.
\qedhere
\end{itemizesmall}
\end{proof}

\section{Monoidal and symmetric monoidal model $\infty$-categories}\label{section mon and symm mon model infty-cats}

In this section, we show that the localization of a (resp.\! \textit{symmetric}) \textit{monoidal model $\infty$-categories} is canonically \textit{closed} (resp.\! \textit{symmetric}) \textit{monoidal}.  For a more leisurely discussion of monoidal and symmetric monoidal model categories, we again refer the reader to \cite[\sec 4.2]{HoveyModelCats}.

\begin{defn}\label{defn of monoidal model infty-cat}
Let $\V \in \Alg(\Cati)$ be a closed monoidal $\infty$-category, and suppose that $\V$ is equipped with a model structure.  We say that these data make $\V$ into a \bit{monoidal model $\infty$-category} if they satisfy the following evident $\infty$-categorical analogs of the usual axioms for a monoidal model category.
\begin{enumeratesmall}
\item[\monppaxiom] \textit{(pushout product)} The underlying two-variable adjunction
\[ \twovaradjV \]
is a two-variable Quillen adjunction.
\item[\monunitaxiom] \textit{(unit)} There exists a cofibrant replacement $\es_\V \cofibn \bbQ \unit_\V \we \unit_\V$ such that the functors
\[ \V \xra{(\bbQ \unit_\V \ra \unit_\V ) \otimes -} \Fun([1],\V) \]
and
\[ \V \xra{- \otimes (\bbQ \unit_\V \ra \unit_\V )} \Fun([1],\V) \]
take cofibrant objects to weak equivalences.
\end{enumeratesmall}
\end{defn}

\begin{rem}\label{unit axiom vacuous when unit cofibt}
The unit axiom {\monunitaxiom} is automatically satisfied whenever the unit object $\unit_\V \in \V$ is itself cofibrant.
\end{rem}


We have the following key example.

\begin{ex}\label{ex KQ is a monoidal model str}
The model $\infty$-category $s\S_\KQ$ of \cref{sspaces:kan--quillen model structure on sspaces} is a monoidal model $\infty$-category with respect to its cartesian symmetric monoidal structure:
\begin{itemizesmall}
\item that the underlying two-variable adjunction is a Quillen adjunction follows from (an identical argument to) the proof of \cite[Lemma 4.2.4]{HoveyModelCats} (see \cite[Corollary 4.2.5]{HoveyModelCats}), and
\item the unit object $\pt_{s\S} \simeq \Delta^0 \in s\S_\KQ$ is cofibrant.
\end{itemizesmall}
\end{ex}

We then have the following result.

\begin{prop}\label{fund thm of monoidal model infty-cats}
Suppose that $\V$ is a monoidal model $\infty$-category.  Then the derived two-variable adjunction of its underlying two-variable Quillen adjunction itself underlies a canonical closed monoidal structure on its localization $\loc{\V}{\bW}$.
\end{prop}

\begin{proof}
Observe that the monoidal product preserves cofibrant objects.  Hence, the underlying non-unital monoidal structure on $\V$ restricts to one on $\V^c$.  Moreover, the structure maps for $\V^c \in \Alg^\nonu(\Cati)$ preserve weak equivalences by Kenny Brown's lemma (\ref{kenny brown}), so we obtain a natural lift to $\V^c \in \Alg^\nonu(\RelCati)$.

Now, the localization functor is symmetric monoidal by \cref{rnerves:localization preserves finite products}, so that we obtain $\loc{\V^c}{(\bW^c)} \in \Alg^\nonu(\Cati)$.  To see that this can in fact be canonically promoted to a unital monoidal structure, we use the guaranteed cofibrant replacement $\es_\V \cofibn \bbQ \unit_\V \we \unit_\V$.  First of all, by assumption, the resulting natural transformations $(\bbQ \unit_\V \otimes -) \ra (\unit_\V \otimes -)$ and $(- \otimes \bbQ \unit_\V) \ra (- \otimes \unit_\V)$ in $\Fun(\V,\V)$ restrict to natural weak equivalences in $\Fun(\V^c,\V)$.  As the unit object comes equipped with equivalences
\[ (\unit_\V \otimes -) \simeq \id_\V \simeq (- \otimes \unit_\V) , \]
it follows that the restrictions along $\V^c \subset \V$ of these functors all lie in the full subcategory
\[ \Fun(\V^c,\V^c)^\Rel \subset \Fun(\V^c,\V^c) \subset \Fun(\V^c,\V) ,\]
where they give rise to a diagram
\[ (\bbQ \unit_\V \otimes - ) \we (\unit_\V \otimes -) \simeq \id_{\V^c} \simeq (- \otimes \unit_\V) \lwe (- \otimes \bbQ \unit_\V ) \]
of natural weak equivalences.  Applying the canonical functor
\[ \Fun(\V^c,\V^c)^\Rel \ra \Fun(\loc{\V^c}{(\bW^c)} , \loc{\V^c}{(\bW^c)}) \]
of \cref{rnerves:map from rel fctr cat to fctrs betw localizns} then yields a diagram
\[ \left( \bbQ \unit_\V \overset{\bbL}{\otimes} - \right) \xra{\sim} \id_{\loc{\V^c}{(\bW^c)}} \xla{\sim} \left( - \overset{\bbL}{\otimes} \bbQ \unit_\V \right) \]
of natural equivalences.  Thus, the map $\pt_\Cati \xra{\bbQ \unit_\V} \loc{\V^c}{(\bW^c)}$ is a quasi-unit (in the sense of Definition A.5.4.3.5) for the non-unital monoidal $\infty$-category $\loc{\V^c}{(\bW^c)} \in \Alg^\nonu(\Cati)$.  It then follows from Theorem A.5.4.3.8 (and Propositions A.4.1.2.15 and A.5.4.3.2) that there exists a unique refinement $\loc{\V^c}{(\bW^c)} \in \Alg(\Cati)$ to a monoidal $\infty$-category.\footnote{Note that Definition A.5.4.3.5 only requires the \textit{existence} of a quasi-unit; the quasi-unit itself is not part of the data.}

The assertion is now clear: we have exhibited a canonical monoidal structure on $\loc{\V^c}{(\bW^c)} \simeq \loc{\V}{\bW}$ whose underlying monoidal product is precisely the left derived bifunctor of the original monoidal product on $\V$, and the derived bifunctors $\bbR \enrhom_l(-,-)$ and $\bbR \enrhom_r(-,-)$, being participants in the derived two-variable adjunction, have no choice but to define left and right internal hom-objects.
\end{proof}

We also have the following variant.

\begin{defn}\label{defn of symm monoidal model infty-cat}
Let $\V \in \CAlg(\Cati)$ be a closed symmetric monoidal $\infty$-category, and suppose that $\V$ is equipped with a model structure.  We say that these data make $\V$ into a \bit{symmetric monoidal model $\infty$-category} if they make the underlying closed monoidal $\infty$-category $\V \in \Alg(\Cati)$ into a monoidal model $\infty$-category.
\end{defn}

We then have the following corresponding result.

\begin{prop}\label{fund thm of symm monoidal model infty-cats}
Suppose that $\V$ is a symmetric monoidal model $\infty$-category.  Then the derived two-variable adjunction of its underlying two-variable Quillen adjunction itself underlies a canonical closed symmetric monoidal structure on its localization $\loc{\V}{\bW}$.
\end{prop}

\begin{proof}
In light of \cref{fund thm of monoidal model infty-cats}, it only remains to show that the symmetric monoidal structure on $\V$ descends canonically to one on $\loc{\V}{\bW}$ (extending its monoidal structure).  Just as in the proof of that result, the underlying datum $\V \in \CAlg^\nonu(\Cati)$ restricts to give $\V^c \in \CAlg^\nonu(\Cati)$, which admits a natural lift $\V^c \in \CAlg^\nonu(\RelCati)$, and then the fact that the localization functor is symmetric monoidal yields $\loc{\V^c}{(\bW^c)} \in \CAlg^\nonu(\Cati)$.  The existence of a canonical lift $\loc{\V^c}{(\bW^c)} \in \CAlg(\Cati)$ now follows from Corollary A.5.4.4.7.
\end{proof}

\begin{rem}\label{lurie proves mon or symm mon str for cofibt unit}
In the special case that our (resp.\! symmetric) monoidal model $\infty$-category $\V$ has that its unit object is cofibrant, then its localization $\loc{\V}{\bW}$ obtains a canonical (resp.\! symmetric) monoidal structure by Proposition A.4.1.3.4.  However, this result does not alone guarantee a \textit{closed} (resp.\! symmetric) monoidal structure, as does \cref{fund thm of monoidal model infty-cats} (resp.\! \cref{fund thm of symm monoidal model infty-cats}).
\end{rem}

\begin{rem}
Though they presumably exist, we do not pursue any notions of ``$\O$-monoidal model $\infty$-category'' for other $\infty$-operads $\O$ here.
\end{rem}

\begin{rem}
In Definitions \ref{defn of monoidal model infty-cat} \and \ref{defn of symm monoidal model infty-cat}, one could remove the requirement that there exist a suitable cofibrant replacement of the unit object (or even that there exist a unit object at all); then, Propositions \ref{fund thm of monoidal model infty-cats} \and \ref{fund thm of symm monoidal model infty-cats} would admit non-unital variants.
\end{rem}

\section{Enriched model $\infty$-categories}\label{section enr model infty-cats}

In this final section, we show that the localization of a model $\infty$-category that is compatibly \textit{enriched and bitensored} over a closed monoidal model $\infty$-category is itself enriched and bitensored over the localization of the enriching model $\infty$-category.  For a more leisurely discussion of monoidal and symmetric monoidal model categories, we yet again refer the reader to \cite[\sec 4.2]{HoveyModelCats} (beginning with \cite[Definition 4.2.18]{HoveyModelCats}).

\begin{defn}\label{defn enr model str}
Let $\V \in \Alg(\Cati)$ be a monoidal model $\infty$-category, let $\M \in \RMod_\V(\Cati)$ be a right $\V$-module (with respect to its underlying monoidal $\infty$-category structure) whose underlying action bifunctor extends to a two-variable adjunction
\[ \twovaradjVM , \]
and suppose that $\M$ is equipped with a model structure.  We say that these these data make $\M$ into a \bit{$\V$-enriched model $\infty$-category} (or simply a \bit{$\V$ model $\infty$-category}) if they satisfy the following evident $\infty$-categorical analogs of the usual axioms for an enriched model category.
\begin{enumeratesmall}
\item[\enrppaxiom] \textit{(pushout product)} The above two-variable adjunction is a two-variable Quillen adjunction.
\item[\enrunitaxiom] \textit{(unit)} There exists a cofibrant replacement $\es_\V \wcofibn \bbQ \unit_\V \we \unit_\V$ such that the functor
\[ \M \xra{- \tensoring (\bbQ \unit_\V \ra \unit_\V)} \Fun([1],\M) \]
takes cofibrant objects to weak equivalences.
\end{enumeratesmall}
We use the same terminology in the case that $\V \in \CAlg(\Cati)$ is in fact a symmetric monoidal model $\infty$-category.
\end{defn}

\begin{defn}\label{defn sspatial model str}
As a special case of \cref{defn enr model str}, we refer to a $s\S_\KQ$-enriched model $\infty$-category as a \bit{simplicio-spatial model $\infty$-category} (recall \cref{ex KQ is a monoidal model str}).
\end{defn}

\begin{ex}
Given a model $\infty$-category $\M$, the resolution model $\infty$-category $s\M_\res$ (see \cref{sspaces:ex resolution model str}) is simplicio-spatial, much as the classical resolution model structure is simplicial (see \cite[3.1 and 5.3]{DKS-E2}).
\end{ex}

\begin{ex}
If $\C_\triv$ is an $\infty$-category equipped with the trivial model structure (see \cref{sspaces:trivial model structure}) and the underlying $\infty$-category $\C$ is bitensored, then $\C_\triv$ can be considered as a simplicio-spatial model $\infty$-category in which
\begin{itemizesmall}
\item the co/tensoring over $s\S$ is obtained by precomposition with $|{-}| : s\S \ra \S$, and
\item the internal hom is obtained by postcomposition with $\const : \S \hookra s\S$.
\end{itemizesmall}
\end{ex}

\begin{ex}
If $\M$ is a simplicial model category (i.e.\! a $s\Set_\KQ$-enriched model category), then $\M$ can also be considered as a simplicio-spatial model $\infty$-category in which
\begin{itemizesmall}
\item the co/tensoring over $s\S$ is obtained by precomposition with $\pi_0^\lw : s\S \ra s\Set$, and
\item the internal hom is obtained by postcomposition with $\disc^\lw : s\Set \hookra s\S$.
\end{itemizesmall}
\end{ex}

\begin{ex}
if $\M$ is a simplicio-spatial model $\infty$-category, then the levelwise action $s\M \tensoring^\lw s\S \ra s\M$ given by $(x_\bullet \tensoring Y)_n = x_n \tensoring Y$ makes $s\M_\Reedy$ into a simplicio-spatial model $\infty$-category.
\end{ex}

We now show that the structure of an enriched model $\infty$-category descends to localizations as claimed.

\begin{prop}\label{fund thm of enriched model infty-cats}
Suppose that $\M$ is a $\V$-enriched model $\infty$-category.  Then the derived two-variable adjunction of its underlying two-variable Quillen adjunction itself underlies a canonical enrichment and bitensoring of $\loc{\M}{\bW_\M}$ over $\loc{\V}{\bW_\V}$.
\end{prop}

\begin{proof}
The proof is almost identical to that of \cref{fund thm of monoidal model infty-cats}, only now we replace the appeal to Theorem A.5.4.3.8 with an appeal to (the dual of) Proposition A.5.4.3.16.
\end{proof}

\begin{rem}
Let $\M$ be a simplicio-spatial model $\infty$-category.  As being bitensored over $\S$ is actually a \textit{condition} (rather than additional structure), it follows that the derived bitensoring over $\loc{s\S}{\bW_\KQ} \simeq \S$ of $\loc{\M}{\bW}$ guaranteed by \cref{fund thm of enriched model infty-cats} must indeed be a bitensoring in the usual sense.
\end{rem}

\bibliographystyle{amsalpha}
\bibliography{qadjns}{}

\end{document}